\documentclass[11pt,a4paper]{article}
\usepackage{amsmath}

\usepackage{epsf,epsfig,amsfonts,amsgen,amsmath,amstext,amsbsy,amsopn,amsthm
}
\usepackage{ebezier,eepic}
\usepackage{color}
\usepackage{multirow}
\newtheorem{thm}{Theorem}[section]

\newtheorem{lem}{Lemma}
\newtheorem{false statement}{False statement}

\theoremstyle{definition}
\newtheorem{definition}{Definition}

\newtheorem{claim}{Claim}

\newtheorem{prob}{Problem}

\def\pf{\noindent {\emph{Proof.}}~~}

\def\qee{ \hfill $\square$}

\newcounter{mathitem}
  {\begin{list}{{$(\roman{mathitem})$}}{
   \setcounter{mathitem}{0}
   \usecounter{mathitem}
   \setlength{\topsep}{0pt plus 2pt minus 0pt}
   \setlength{\parskip}{0pt plus 2pt minus 0pt}
   \setlength{\partopsep}{0pt plus 2pt minus 0pt}
   \setlength{\parsep}{0pt plus 2pt minus 0pt}
   \setlength{\leftmargin}{35pt}
   \setlength{\itemsep}{0pt plus 2pt minus 0pt}}}
  {\end{list}}

\begin{document}

\title{\bf Graph operations and a unified method for kinds of Tur\'an-type problems on paths, cycles and matchings}

\date{}

\author{Jiangdong Ai\thanks{School of Mathematical Sciences and LPMC, Nankai University, Tianjin 300071, P.R.
China. Email: jd@nankai.edu.cn. Partially supported
by the Fundamental Research Funds for the Central Universities, Nankai University. },~ Hui Lei\thanks{School of Statistics and Data Science, LPMC and KLMDASR, Nankai University, Tianjin 300071, P.R.
China. Email: hlei@nankai.edu.cn. Partially supported
by the NSFC grant (No.\ 12371351).},~ Bo
Ning\thanks{Corresponding author. College of Computer Science, Nankai University, Tianjin 300350, P.R.
China. Email: bo.ning@nankai.edu.cn. Partially supported
by the NSFC grant (No.\ 11971346).},~ Yongtang Shi\thanks{Center for Combinatorics and LPMC, Nankai University, Tianjin 300071, P.R.
China. Email: shi@nankai.edu.cn. Partially supported
by the NSFC grant (No.\ 12161141006).}}
\maketitle

\begin{center}
\begin{minipage}{140mm}
\small\noindent{\bf Abstract:}
Let $G$ be a connected graph and $\mathcal{P}(G)$ a graph parameter. We say that $\mathcal{P}(G)$ is feasible if $\mathcal{P}(G)$ satisfies the following properties:
(I) $\mathcal{P}(G)\leq \mathcal{P}(G_{uv})$, if $G_{uv}=G[u\to v]$ for any $u,v$, where $G_{uv}$ is the graph obtained by applying Kelmans Operation from $u$ to $v$;
(II) $\mathcal{P}(G) <\mathcal{P}(G+e)$ for any edge $e\notin E(G)$. Let $P_k$ be a path of order $k$, $\mathcal{C}_{\geq k}$ the set of all cycles of length at least $k$ and $M_{k+1}$ a matching of $k+1$ independent edges. In this paper, we mainly prove the following three results:\\
(i) Let $n\geq k\geq 5$ and let $t=\left\lfloor\frac{k-1}{2}\right\rfloor$.
Let $G$ be a $2$-connected
$n$-vertex $\mathcal{C}_{\geq k}$-free graph with the maximum $\mathcal{P}(G)$ where $\mathcal{P}(G)$ is feasible. Then, $G\in \mathcal{G}^1_{n,k}=\{W_{n,k,s}=K_{s}\vee ((n-k+s)K_1\cup K_{k-2s}): 2\leq s\leq t\}$.\\
(ii) Let $n\geq k\geq 4$ and let $t=\left\lfloor\frac{k}{2}\right\rfloor-1$.
Let $G$ be a connected
$n$-vertex $P_{k}$-free graph with the maximum $\mathcal{P}(G)$ where $\mathcal{P}(G)$ is feasible. Then, $G\in \mathcal{G}^2_{n,k}=\{W_{n,k-1,s}=K_{s}\vee ((n-k+s+1)K_1\cup K_{k-2s-1}): 1\leq s\leq t\}.$\\
(iii)
Let $G$ be a connected
$n$-vertex $M_{k+1}$-free graph with the maximum $\mathcal{P}(G)$ where $\mathcal{P}(G)$ is feasible. Then, $G\cong K_n$ when $n=2k+1$ and $G\in \mathcal{G}^3_{n,k}=\{K_s\vee  ((n-2k+s-1)K_1\cup K_{2k-2s+1}):1\leq s\leq k\}$ when $n\geq 2k+2$. 

Directly derived from these results, we obtain a series of applications in Tur\'an-type problems, generalized Tur\'an-type problems, powers of graph degrees in extremal graph theory, and problems related to spectral radius and signless Laplacian spectral radius in spectral graph theory. Our results generalize classical results on cycles and matchings due to Kopylov and Erd\H{o}s-Gallai, respectively, and provide a positive resolution to an open problem originally proposed by Nikiforov. We improve and extend the spectral extremal results on paths due to Nikiforov, and due to Nikiforov and Yuan. We also offer a comprehensive solution to a connected version of a problem on the degree power sum of a graph containing no $P_k$, a topic initially studied by Caro and Yuster.    

\end{minipage}
\end{center}


\section{Introduction}\label{S1}
The main goal of this paper is to develop a method that provides a unified approach for solving some Tur\'an-type and generalized Tur\'an-type problems, degree power problems, and extremal spectra problems (mainly under spectral radius conditions and signless Laplacian spectral radius conditions) on paths, cycles, and matchings. Before our work, all topics mentioned here have almost exclusively been studied independently and through distinct approaches. In the following, we will give a brief review of each of these topics.

\subsection{Tur\'an-type and generalized Tur\'an-type theorems on paths, cycles and matchings}
A fundamental result in graph theory asserts that any graph with $n$ vertices and $m\geq n$ edges has a cycle. Strengthening this fact, a cornerstone result attributed to Erd\H{o}s and Gallai \cite{EG59} says that
if an $n$-vertex graph has at least $m\geq n$ edges then there is a cycle of length at least $\frac{2m}{n-1}$. Given a family of graphs $\mathcal{H}$, we denote the \emph{Tur\'{a}n number} of $\mathcal{H}$ by $ex(n,\mathcal{H})$, that is the maximum number of edges in an $n$-vertex graph which contains no $H$ as a subgraph for each $H\in \mathcal{H}$. When $\mathcal{H}=\{H\}$, we use $ex(n,H)$ instead of $ex(n,\mathcal{H})$. In this language, Erd\H{o}s-Gallai Theorem is equivalent to that $ex(n,{\mathcal{C}}_{\geq k})=\frac{(k-1)(n-1)}{2}$, where $\mathcal{C}_{\geq k}$ is the set of all cycles of lengths at least $k$, $3\leq k\leq n$.

Since the extremal graph contains cut-vertices,
the classical Erd\H{o}s-Gallai Theorem can be improved
if we assume that $G$ is $2$-connected. Under this setting, among the improvements due to Woodall \cite{WD1976}, Lewin \cite{LM75}, Faudree and Schelp \cite{FS75}, and Kopylov \cite{K77}, the following theorem on cycles due to Kopylov stands out as the most robust one in certain aspects.

Define $W_{n,k,s}$ to be a graph on $n$ vertices, in which its vertex set can be partitioned into three subsets $X,Y,Z$, in which
$|X|=s$, $|Y|=k-2s$, $|Z|=n-(k-s)$, and the edge set consists of all possible edges between $X$ and $Z$ and all edges in $X\cup Y$. The graphs $W_{n,k,2}$ and $W_{n,k,t}$ show that Kopylov's theorem is sharp.

\begin{thm}[Kopylov \cite{K77}]\label{K77}
Let $n\geq k\geq 5$ and let $t=\left\lfloor\frac{k-1}{2}\right\rfloor$.
If $G$ is a 2-connected
$n$-vertex graph with $e(G)>\max\{e(W_{n,k,2}),e(W_{n,k,t})\}$,
then $G$ has a cycle of length at least $k$.
\end{thm}

The theorem of path version is similar to Theorem \ref{K77}, where equality case was determined by Balister,  Gy\H{o}ri, Lehel, and Schelp
in \cite{BGLS08}.

\begin{thm}[Kopylov \cite{K77}, Balister,  Gy\H{o}ri, Lehel, and Schelp \cite{BGLS08}]\label{K77-2}
Let $n\geq k\geq 4$ and $t=\left\lfloor\frac{k}{2}\right\rfloor-1$.
If $G$ is a connected graph on $n$ vertices with
$$
e(G)\geq \max\left\{\binom{k-2}{2}+n-k+2,\binom{\lceil\frac{k}{2}\rceil}{2}+(\left\lfloor\frac{k}{2}\right\rfloor-1)\left(n-\left\lceil\frac{k}{2}\right\rceil\right)\right\},
$$
then $G$ contains a path of order $k$, unless $G$ is either $W_{n,k-1,1}$
or $W_{n,k-1,t}$.
\end{thm}

Another important theorem in extremal graph theory is Erd\H{o}s-Gallai's matching theorem as follows, in which $M_{k}$ is a matching of $k$ independent edges, and ``$\vee$" is the join operation. Akiyama and Frankl \cite{AF85} gave a short proof of Erd\H{o}s-Gallai Theorem by the shifting method. 
\begin{thm}[Erd\H{o}s-Gallai \cite{EG59}]
If $G$ is an $n$-vertex graph with
$$
e(G)\geq \max \left\{\binom{2k+1}{2}, \binom{n}{2}-k(n-k)\right\},
$$
then $G$ contains an $M_{k+1}$, unless $G=K_{2k+1}$ or $G=K_k\vee ((n-k)K_1)$.
\end{thm}
Rather than determine the maximum number of edges, Alon and Shikhelman \cite{AS16} studied the function $ex(n,F,H)$ which means the maximum number of copies of $F$ in an $H$- free $n$-vertex graph. Such type of problem is called \emph{Generalized Tur\'{a}n Problem}. Note that when $F=K_2$, we are back to the classic \emph{Tur\'{a}n Problem}.
Let $N_s(G)$ denote the maximum number of unlabeled copies of $K_s$ in $G$. By using Kopylov's technique, Luo \cite{L18} extended Kopylov's theorem on cycles and paths to its clique versions. Gy\H{o}r, Salia, Tompkins, and Zamora \cite{GSTZ19} obtained some extensions by counting stars. As we see later, the numbers of such certain subgraphs are defined as weakly-feasible parameters by us in this paper, we can give a unified method to deal with them.   

Very recently, Theorem \ref{K77} has received more and more attentions, see \cite{L21,L18,MN20}. The stability form of Kopylov's theorem seems to play an important role in solving numerous conjectures and problems from different aspects of graph theory, for example, classical Anti-Ramsey conjecture and problems in spectral graph theory, see \cite{LL21, LN23, Y21+}.

\subsection{Degree power problems on paths and cycles}
Given a graph $G$ of order $n$ with degree sequence $d_1,d_2\dots d_n$, we use $\sum_{i=1}^{n}d_i^p$ where $p\geq 1$ to denote the \emph{degree power} of $G$. Note that when $p=1$, $\sum_{i=1}^{n}d_i=2e(G)$. People 
are captivated by the exploration of the maximum degree power of graphs that involve forbidden subgraphs. Our emphasis is specifically directed towards restraining long paths and cycles in this investigation. 

Caro and Yuster~\cite{CY00} demonstrated the maximum degree power in graphs that are $P_k$-free assuming that $n$ is considerably larger than $k$, where $P_k$ denotes the path of order $k$. They also provided a characterization of the extremal graphs. To obtain the concrete magnitude relationship between $n$ and $k$ seems to be difficult. They guessed that $n=\Omega(k^2)$. In this paper, we give a solution to the connected version of Caro-Yuster's problem.

\subsection{Spectral extrema on paths and cycles}
Let $G$ be a graph and $A:=A(G)$ be its adjacency matrix. Let $\lambda_1\geq \lambda_2\geq... \geq \lambda_n$ be all eigenvalues of $A$. The spectral radius of $G$, denoted by $\lambda(G)$, is defined to be $\max\{|\lambda_i|:1\leq i\leq n\}$. The signless Laplacian spectral radius of $G$, denoted by $q(G)$, is the largest eigenvalue of $A(G)+D(G)$, where $D(G)$ is the degree matrix of $G$. For a given graph $H$, we denote by $spex(n,H)$ ($spex_{c}(n,H)$, $spex_{2-con}(n,H)$) the maximum spectral radius of a (connected, 2-connected) graph on $n$ vertices which contains no $H$ as a subgraph. We use $sspex(n,H)$ ($sspex_c(n,H)$)
to denote the maximum signless spectral radius of a (connected) graph on $n$ vertices which contains no $H$ as a subgraph. 

In contrast to classical extremal graph theory, spectral extremal graph theory is a relatively young and active branch. For a survey on its developments and open problems, we refer to \cite{N11}.  
Till now, very little is known about complete pictures of certain classes of graphs, especially for some classes in which the order can turn to infinity. 

In this paper, we focus on the spectra extrema of graphs without a given length of paths or with a given circumference.  We denote $K_k\vee (n-k)K_1$ by $S_{n,k}$, and $K_k\vee ((n-k-2)K_1\cup K_2)$ by $S^+_{n,k}$ for simplicity. This problem can be dated back at least to \cite{N10}.  Nikiforov \cite{N10} proved that $spex(n, P_{2k+2})=\lambda(S_{n,k})$ 
and $spex(n, P_{2k+3})=\lambda(S^+_{n,k})$ for $n\geq 2^{4k}$ and $k\geq 1$. Improving Nikiforov's result, Gao and Hou \cite{GH19}
proved that $spex(n,\mathcal{C}_{\geq 2k+1})=\lambda(S_{n,k})$
and $spex(n,\mathcal{C}_{\geq 2k+2})=\lambda(S^+_{n,k})$ for $k\geq 2$ and $n\geq13k^2$. One may wonder about the case when $k=\Theta(n)$. In particular, when $G$ is spanning or nearly spanning,
what do we know? For this case, Fiedler and Nikiforov \cite{FN10} proved that $spex(n,P_{n})=\lambda(K_{n-1}\cup K_1)$. A natural problem is to determine all the values of $spex(n, P_k)$ for any $4\leq k\leq n$, which was suggested by Prof. V. Nikiforov in an email to the third author when he was a Ph.D. student several years ago.
As we know, this difficult problem is wide open till now. 

In this paper, we study the following two problems. The first one includes Nikiforov's problem.

\begin{prob}\label{prob:spex_c}
Let $k\geq 4$ and $n\geq k$  be two integers.
Determine the function $spex_{c}(n,P_{k})$ and $spex(n,P_k)$.
\end{prob}

\begin{prob}\label{prob:spex_c_2con}
Let $k\geq 5$ and $n\geq k$  be two integers. Determine the function $spex_{2-con}(n,\mathcal{C}_{\geq k})$.
\end{prob}
We solved these two problems above completely, i.e., our results cover all possible values of $n$ and $k$. We would like to point out that this is the first time introducing the Kopylov-style technique to spectral graph theory. In addition, we also prove an analog for graphs with a given matching number.

For the signless Laplacian spectral version of problems on paths, Nikiforov and Yuan \cite{NY14} proved that for $k\geq 1$, $n\geq 7k^2$ and an $n$-vertex graph $G$, if $q(G)>q(S_{n,k})$, then $G$ contains a $P_{2k+2}$; if $q(G)>q(S^+_{n,k})$, then $G$ contains a $P_{2k+3}$. In this paper, we also give solutions to the signless Laplacian spectral radius versions of Problems \ref{prob:spex_c} and \ref{prob:spex_c_2con}.

\subsection{Terminology and notation}
Let $G=G(V,E)$ be a graph and $X\subseteq V$. The subgraph of $G$ induced by $X$ is denoted by $G[X]$. The vertices in $N_G(v)$ are neighbors of $v$. Then, $d_G(v)=|N_G(v)|$. Denote $N_G(v)\cup \{v\}$ by $N_G[v]$. We may omit the subscripts sometimes when it is clear from the context.
Let $G_1$ and $G_2$ be two graphs. We use $G_1\cup G_2$ to denote the disjoint union of $G_1$ and $G_2$, and $G_1\vee G_2$ to denote the join of $G_1$ and $G_2$, i.e., besides the edges in $G_1\cup G_2$, it contains all possible edges from $G_1$ to $G_2$. We use $K_1$ to denote an isolated vertex, and $kK_1$ to denote
an isolated set of $k$
vertices. For a subset $S\subseteq V(G)$, we denote by $G-S$ the subgraph of $G$ induced by $V(G)\backslash V(S)$.

We use $c(G)$ and $\nu(G)$ to denote the length of a longest cycle and the size of a maximum matching, in $G$, respectively.
Given two sequences $d=(d_1,d_2,\dots,d_n)$ and $c=(c_1,c_2,\dots,c_n)$ in decreasing order, i.e, in which $d_i\geq d_j$ and $c_i\geq c_j$ when $i<j$. We call $d>c$ if there exists a $k\in [n]=\{1,2,\dots,n\}$ such that $d_k> c_k$ and $d_i=c_i$ for all $1\leq i<k$.

\subsection{Outline}
This paper is organized as follows. In the first section, we provide a concise overview of the background of the relevant issues. Subsequently, in Section \ref{Sec:2}, we will briefly outline the contributions made by our research. In Section~\ref{fsec} and Section~\ref{imsec}, we will elaborate on applications of our research results across various problem scenarios. Specifically, we systematically elucidate the unified representation of different graph parameters under the conditions of prohibiting long cycles or restricting long paths. In Section~\ref{pfsec}, we present the proofs of the main theorems along with necessary lemmas and tools. In the final section, we provide some comments.

\section{Our contributions}\label{Sec:2}

All the graph parameters mentioned in Section~\ref{S1} are indeed applicable to a connected graph. Surprisingly, we have observed similar behaviors among these parameters from certain perspectives. Therefore, we aim to abstract them into a general type of graph parameters and analyze them systematically.
The members in this type of graph parameters, called feasible graph parameters, include edge number, spectral radius, signless Laplacian spectral radius, degree power, and the number of $s$-cliques or $K_{1,r}$, etc.

Our new concept depends on Kelmans Operation, which was introduced by Kelmans in \cite{K81}. We write the definition as follows.
(The original form of Kelmans Operation does not care about the adjacency of two given vertices. However, we also divide it into two types of operations according to the adjacency for the sake of use later.)

\begin{definition}
(i) Let $G$ be a graph and $xy\in E(G)$. The graph after a Kelmans Operation (KO in short) of $G$ from $x$ to $y$, is denoted by $G'=G[x\rightarrow y]$, if $G'$ satisfies $V(G')=V(G)$ and $E(G')=(E(G)\backslash \{wx:w\in N_G(x)\backslash N_G[y]\})\cup \{wy:w\in N_G(x)\backslash N_G(y)\}$;\\
(ii) Let $G$ be a graph and $xy\notin E(G)$.
The graph after an extended Kelmans Operation (EKO in short) of $G$ from $x$ to $y$, is denoted by $G'=G[x\rightarrow y]$, if $G'$ satisfies $V(G')=V(G)$ and $E(G')=(E(G)\backslash \{wx:w\in N_G(x)\backslash N_G[y]\})\cup \{wy:w\in N_G(x)\backslash N_G(y)\}$.
\end{definition}

Kelmans Operation plays an important role in solving open problems on long paths and cycles. With the help of Klemans Operation, Li and the third author \cite{LN21} confirmed Woodall's conjecture in 1975 which states that every $2$-connected graph on $n$ vertices has a cycle of length at least $2k$ if the number of vertices with degree at least $k$ is at least $\frac{n}{2}+k$, and also characterzied the extremal graphs which have the maximum spectral radius among all non-Hamiltonian graphs with sufficiently larger order $n$ and given minimum degree in \cite{LN16}.
Inspired by these two works, we introduce the following new concept aimed at unifying various phenomenons on paths, cycles, and matchings in both extremal and spectral graph theory.

A \emph{graph parameter} is a function $\phi:\mathbb{G}\to \mathbb{R}$ where $\mathbb{G}$ is the set of finite graphs and $\mathbb{R}$ is the set of real numbers.

\begin{definition}\label{basic}
Let $G$ be a connected graph and $\mathcal{P}(G)$ a graph parameter. We say that $\mathcal{P}(G)$ is feasible, if $\mathcal{P}(G)$ satisfies the following properties:\\ 
(I) $\mathcal{P}(G)\leq \mathcal{P}(G_{uv})$, if $G_{uv}=G[u\to v]$ for any vertices $u,v$;\\
(II) $\mathcal{P}(G) <\mathcal{P}(G+e)$, if a new edge $e\notin E(G)$ but $V(e)\cap V(G)\neq \emptyset$ (here the edge $e$ maybe contains some new vertex other than $V(G)$).
\end{definition}
If $H$ is a connected proper subgraph of $G$, then $\mathcal{P}(H)<\mathcal{P}(G)$ as we can apply (II) of Definition~\ref{basic} repeatedly. 

By combining the Kopylov-type technique with Kelmans Operation (together with some structural analysis), we prove the following three main results.
The proofs are postponed to Section \ref{pfsec}.

\begin{thm}\label{Thm:Main1}
Let $n\geq k\geq 5$ and let $t=\left\lfloor\frac{k-1}{2}\right\rfloor$.
Let $G$ be a $2$-connected
$n$-vertex $\mathcal{C}_{\geq k}$-free graph with the maximum $\mathcal{P}(G)$ where $\mathcal{P}(G)$ is feasible. Then, $G\in \mathcal{G}^1_{n,k}=\{W_{n,k,s}=K_{s}\vee ((n-k+s)K_1\cup K_{k-2s}): 2\leq s\leq t\}$. 
\end{thm}

\begin{thm}\label{Thm:Main2}
Let $n\geq k\geq 4$ and let $t=\left\lfloor\frac{k}{2}\right\rfloor-1$.
Let $G$ be a connected
$n$-vertex $P_{k}$-free graph with the maximum $\mathcal{P}(G)$ where $\mathcal{P}(G)$ is feasible. Then, $G\in \mathcal{G}^2_{n,k}=\{W_{n,k-1,s}=K_{s}\vee ((n-k+s+1)K_1\cup K_{k-2s-1}): 1\leq s\leq t\}.$
\end{thm}

\begin{thm}\label{Thm:Matching}
Let $G$ be a connected
$n$-vertex $M_{k+1}$-free graph with the maximum $\mathcal{P}(G)$ where $\mathcal{P}(G)$ is feasible. Then, $G\cong K_n$ when $n=2k+1$ and $G\in \mathcal{G}^3_{n,k}=\{W_{n,2k+1,s}=K_s\vee  ((n-2k+s-1)K_1\cup K_{2k-2s+1}):1\leq s\leq k\}$ when $n\geq 2k+2$.
\end{thm}
Note that for each main theorem above, if we replace the condition that ``$\mathcal{P}(G)$ is feasible" with ``$\mathcal{P}(G)$ is weakly-feasible" (one can find this concept in Section \ref{fsec}), then we can only say that there exists at least one extremal graph in $\mathcal{G}^i_{n,k}$ ($1\leq i\leq 3$) but not all. In other words, we can just determine the corresponding extremal values. 

Based on these theorems, we can easily give alternative proofs of some known results mentioned in Section~\ref{S1}. In addition, it fully or partially solved some open problems. We believe this represents the essence of a certain type of problems.

\section{On feasible parameters of graphs}\label{fsec}
In this section, we show that several classical graph parameters of a connected graph, including size, degree power, spectral radius, and signless Laplacian spectral radius, are feasible.



\begin{thm}\label{dpf}
Let $G$ be a connected graph
and let $p\geq 2$ be a real. The degree power $\sum_{v\in V(G)}d^p(v)$ of $G$
is feasible.
\end{thm}
\begin{proof}
 Denote $\sum_{v\in V(G)}d^p(v)$ of $G$ by $D(G)$. Given two vertices $u$ and $v$. Let $x=|N(u)\backslash N(v)|$, $y=|N(v)\backslash N(u)|$ and $z=|N(u)\cap N(v)|$. Let $G_{uv}:=G[u\rightarrow v]$. We have that $$D(G_{uv})=D(G)-((x+z)^p+(y+z)^p)+((x+y+z)^p+z^p).$$ Here, we assume $x+y>0$. Otherwise, we can easily get $D(G_{uv})=D(G)$.
 Let $f(x)=x^p$. Then $f(x)$ is a strictly convex function. By Jensen Inequality, we have that 
 \begin{align}
   \frac{x}{x+y}f(z)+\frac{y}{x+y}f(x+y+z)\geq&f(y+z).\label{in1}\\  
\frac{y}{x+y}f(z)+\frac{x}{x+y}f(x+y+z)\geq&f(x+z).\label{in2}  
 \end{align}
Summing (\ref{in1}) and (\ref{in2}), we have $$f(z)+f(x+y+z)\geq f(x+z)+f(y+z).$$
Thus, $D(G_{uv})\geq D(G)$. Obviously, $D(G+e)>D(G)$ for any $e\notin E(G)$. So, $\sum_{v\in V(G)}d^p(v)$ is feasible.
\end{proof}

\begin{thm}
The size of a connected graph is feasible.
\end{thm}
\begin{proof}
It is obviously Kelmans operation keeps the number of edges in $G$, and $e(G+uv)>e(G)$
for any new edge $uv$. This proves the theorem.
\end{proof}

The following two lemmas are from Zhan \cite{Z13}.
\begin{lem}[Zhan \cite{Z13}]
Let $A$ and $B$ be two nonnegative square matrices. If $B<A$ and $A$ is irreducible,
then $\lambda(B) <\lambda(A)$.  
\end{lem}

\begin{lem}[Zhan \cite{Z13}]
Let $A$ be a nonnegative square matrix. If $B$ is a principal submatrix of $A$, then $\lambda(B)\leq \lambda(A)$. If $A$ is irreducible and $B$ is a proper principal submatrix of A, then $\lambda(B)<\lambda(A)$.
\end{lem}

Let $G$ be a graph on $n$ vertices, and let $A(G)$ be the adjacency matrix of $G$. Thus, $G$ is connected if and only if $A(G)$ is irreducible.

The following are basic facts which can be deduced from above.

\begin{lem}\label{Lem:lambda+q_addedge}
Let $G$ be a connected graph and $xy\notin E(G)$. Suppose $G+xy$ is connected.
Then (i) $\lambda(G+xy)>\lambda(G)$; and
(ii) $q(G+xy)>q(G)$.
\end{lem}

\begin{lem}\label{Lem:lambda+q_proper}
Let $G$ be a connected graph and let $H$ be a proper subgraph
of $G$.
Then (i) $\lambda(G)>\lambda(H)$; and 
(ii) $q(G)>q(H)$.
\end{lem}

\begin{lem}(\cite{C09})\label{Lem:C09}
Let $G$ be a connected graph and $u,v\in V(G)$ (maybe $uv\notin E(G)$). Let $G_{uv}:=G[u\rightarrow v]$. Then
$\lambda(G_{uv})\geq \lambda(G)$.
\end{lem}

\begin{lem}(\cite{LN16})\label{Lem:LN16}
Let $G$ be a connected graph and $u,v\in V(G)$ (maybe $uv\notin E(G)$). Let $G_{uv}:=G[u\rightarrow v]$. Then
$q(G_{uv})\geq q(G)$.
\end{lem}

Thus, by the above lemmas, we have the following. 
\begin{thm}\label{srf}
The spectral radius of a connected graph is feasible.
\end{thm}
\begin{proof}
The theorem follows from Lemmas~\ref{Lem:C09},~\ref{Lem:lambda+q_addedge}(i) and \ref{Lem:lambda+q_proper}(i).
\end{proof}

\begin{thm}\label{lsrf}
The signless Laplacian spectral radius of a connected graph is feasible.
\end{thm}
\begin{proof}
The theorem follows from Lemmas~\ref{Lem:LN16},~\ref{Lem:lambda+q_addedge}(ii) and \ref{Lem:lambda+q_proper}(ii).
\end{proof}

\begin{definition}
Let $G$ be a connected graph and $\mathcal{P}(G)$ a graph parameter. We say that $\mathcal{P}(G)$ is \emph{weakly-feasible}, if $\mathcal{P}(G)$ satisfies the following properties:\\ 
(I) $\mathcal{P}(G)\leq \mathcal{P}(G_{uv})$, if $G_{uv}=G[u\to v]$ for any vertices $u,v\in V(G)$;\\
(II) $\mathcal{P}(G) \leq\mathcal{P}(G+e)$, if a new edge $e\notin E(G)$ but $V(e)\cap V(G)\neq \emptyset$. 
\end{definition}

\begin{thm}\label{Thm:s-clique}
The number of $s$-cliques in a connected graph is weakly-feasible.
\end{thm}
\begin{proof}
Let $G$ be a connected graph and $n_s(G)$ denote the number of $s$-cliques. For any two vertices $u,v\in V(G)$, let $G_{uv}:=G[u\to v]$. For any $s$-clique $K$,
if $K$ contains no $\{u,v\}$ or contains both $u,v$, denote by $K_{uv}=K$; if $K$
contains $u$ but not $v$,
denote by $K_{uv}=G_{uv}[(K\setminus \{u\})\cup\{v\}]$, which is an $s$-clique in $G_{uv}$.

Observe that if $K$ is an $s$-clique not containing $u$ and $v$ in $G$, then $K$ is still an $s$-clique in $G_{uv}$; if $K$ is an $s$-clique containing both $u$ and $v$ in $G$, then $K$ is still an $s$-clique in $G_{uv}$; if $K$ is an $s$-clique containing $u$ but not $v$ in $G$, then $K_{uv}$ is an $s$-clique in $G_{uv}$. For any two distinct $s$-cliques $K,K'\subseteq G$ containing $u$ but not $v$, as $K\neq K'$, $V(K)\neq V(K')$.
Then we can see 
$V(K_{uv})\neq V(K'_{uv})$. This means that $K_{uv}\neq K'_{uv}$ for this case. If $K$ is an $s$-clique containing $v$ but not $u$ in $G$, then $K$ is still an $s$-clique in $G_{uv}$.
Thus, we find a bijection $f: K\rightarrow K_{uv}$,
and so
$n_s(G)\leq n_s(G_{uv})$. 
\end{proof}
Analogously, we can get the following result.
 \begin{thm}\label{star}
The number of $K_{1,r}$ in a connected graph is weakly-feasible.
\end{thm}
\section{Implications}\label{imsec}
All corollaries in this section are assumed Theorems \ref{Thm:Main1}, \ref{Thm:Main2} and \ref{Thm:Matching}, whose proofs are postponed to Section \ref{pfsec}. From the proofs of our main theorems, it is not hard to check that if we replace `feasible' with `weakly-feasible', then we could get the following results accordingly.
\begin{thm}\label{Thm:Main1-2}
Let $n\geq k\geq 5$ and let $t=\left\lfloor\frac{k-1}{2}\right\rfloor$.
Let $G$ be a $2$-connected
$n$-vertex $\mathcal{C}_{\geq k}$-free graph with the maximum $\mathcal{P}(G)$ where $\mathcal{P}(G)$ is weakly-feasible. Then, $\mathcal{P}(G)\leq\max \{\mathcal{P}(W_{n,k,s})=\mathcal{P}(K_{s}\vee ((n-k+s)K_1\cup K_{k-2s})): 2\leq s\leq t\}$. 
\end{thm}

\begin{thm}\label{Thm:Main2-2}
Let $n\geq k\geq 4$ and let $t=\left\lfloor\frac{k}{2}\right\rfloor-1$.
Let $G$ be a connected
$n$-vertex $P_{k}$-free graph with the maximum $\mathcal{P}(G)$ where $\mathcal{P}(G)$ is weakly-feasible. Then, $\mathcal{P}(G)\leq\max \{\mathcal{P}(W_{n,k-1,s})=\mathcal{P}(K_{s}\vee ((n-k+s+1)K_1\cup K_{k-2s-1})): 1\leq s\leq t\}.$
\end{thm}

\begin{thm}\label{Thm:Matching-2}
Let $G$ be a connected
$n$-vertex $M_{k+1}$-free graph with the maximum $\mathcal{P}(G)$ where $\mathcal{P}(G)$ is weakly-feasible. Then, $\mathcal{P}(G)\leq\max \{\mathcal{P}(W_{n,2k+1,s})=\mathcal{P}(K_s\vee  ((n-2k+s-1)K_1\cup K_{2k-2s+1})):1\leq s\leq k,\mathcal{P}(K_{2k+1})\}$.
\end{thm}
\subsection{Tur\'an-type results and generalized Tur\'an-type results}
Based on our main theorems, we can give an alternative proof for each of the following known results.
\begin{thm}[Erd\H{o}s-Gallai \cite{EG59}, Akiyama-Frankl \cite{AF85}]
Let $n\geq 2k+1$. If $G$ is a graph with maximum number of edges such that $G$ is $M_{k+1}$-free, then $G\cong K_{2k+1}\cup (n-2k-1)K_1$ or $G\cong S_{n,k}$.
\end{thm}
\begin{proof}
Suppose that $G$ is connected. If $n\geq 2k+2$, then $|E(G)|=\max\{(n-2k+s-1)s+\binom{2k-s+1}{2}: 1\leq s\leq k\}$ by Theorem~\ref{Thm:Matching}. Let $f(s)=(n-2k+s-1)s+\binom{2k-s+1}{2}$. Observe that $f(s)$ is a convex function of $s$ in $[1,k]$, so $G\cong S_{n,k}$ or $G\cong K_1\vee((n-2k)K_1\cup K_{2k-1})$. Observe that if $k=1$ we have $S_{n,k}\cong K_1\vee((n-2k)K_1\cup K_{2k-1})$. So, we may assume $k\ne 1$. Now, let $f_1(n,t)=|E(S_{n,t})|=nt-\frac{t^2}{2}-\frac{t}{2}$, $f_2(n,t)=|E(K_1\vee((n-2t)K_1\cup K_{2t-1}))|=n+2t^2-3t$ and $f_3(2t+1,t)=|E(K_{2t+1})|=2t^2+t$. By elementary calculus, we have when $2k+1\leq n< 4k$, $f_2(n,k)<f_3(n,k)$; when $n>\frac{5k}{2}$, $f_2(n,k)<f_1(n,k)$; and when $n>\frac{5k}{2}+\frac{3}{2}$, $f_1(n,k)>f_3(n,k)$, which means $|E(K_1\vee((n-2k)K_1\cup K_{2k-1}))|<\max\{|E(K_{2k+1})|,|E(S_{n,k})|\}$ for $n\geq 2k+1$. 

Now, assume that $G_i$ are two connected components of $G$ with order $n_i$ and the matching number $\nu(G_i)=a_i>0$ for $i=1,2$. Since $G$ is edge-maximal, we have $n_i\geq 2a_i+1$; otherwise, we can add one edge $e$ between $G_1$ and $G_2$ without increasing the matching number of $(G_1\cup G_2)+e$. Elementary calculus gives that $f_1(n_1,a_1)+f_1(n_2,a_2)<f_1(n_1+n_2,a_1+a_2)$, $f_3(2a_1+1,a_1)+f_3(2a_2+1,a_2)<f_3(2a_1+2a_2+1,a_1+a_2)$, and $f_1(n_1,a_1)+f_3(2a_2+1,a_2)<f_1(n_1+2a_2+1,a_1+a_2)$. Also, note that if $k$ is fixed, $f_1(n,k)$ will increase as $n$ increases. Now, this completes the proof.
\end{proof}

In the following, for a graph $G$, we use $n_s(G)$ and $s_r(G)$ to denote the number of $s$-cliques and the number of $K_{1,r}$ in $G$, respectively.
\begin{thm}[Luo \cite{L18}]
Let $n\geq k\geq 5$ and let $t=\lfloor\frac{k-1}{2}\rfloor$. If $G$ is a 2-connected $n$-vertex graph with circumference less than $k$, then
$n_s(G)\leq \max\{f_s(n, k,2), f_s(n, k, t)\}$,
where $f_s(n,k,a):=\binom{k-a}{s}+(n-k+a)\binom{a}{s-1}$. 
\end{thm}
\begin{proof}
By Theorem~\ref{Thm:s-clique}, the number of $s$-cliques in $G$ is weakly-feasible. Thus, $n_s(G)\leq \max\{f_s(n,k,a):2\leq a\leq t\}$. As $f_s(n,k,a)$ is convex, $n_s(G)\leq \max\{f_s(n,k,2),f_s(n,k,t)\}$ by Theorem \ref{Thm:Main1-2}. This proves the theorem.
\end{proof}
Similarly, we can prove the following.
\begin{thm}[Luo \cite{L18}]
Let $n\geq k\geq 4$ and let $t=\left\lfloor\frac{k}{2}\right\rfloor-1$.
If $G$ is a connected
$n$-vertex graph with $n_s(G)>\max\{n_s(W_{n,k-1,1}),n_s(W_{n,k-1,t})\}$,
then $G$ has a path of order $k$.
\end{thm}
 Theorem~\ref{star} and Theorem~\ref{Thm:Main2-2}, together with some similar discussions give us the following result, which is a connected version of Gy\H{o}ri-Salia-Tompkins-Zamora's result~\cite{GSTZ19}.
\begin{thm}[Gy\H{o}ri-Salia-Tompkins-Zamora \cite{GSTZ19}]
Let $n\geq k\geq 4$ and $t=\left\lfloor\frac{k}{2}\right\rfloor-1$.
If $G$ is a connected
$n$-vertex graph with $s_r(G)>\max\{s_r(W_{n,k-1,s}:1\leq s\leq t\}$,
then $G$ has a path of order $k$.
\end{thm}

\subsection{Degree power of graphs}
We have some results on the degree power of 2-connected graphs without cycles of length at least $k$.

\begin{thm}
Let $n\geq k\geq 5$ and let $t=\left\lfloor\frac{k-1}{2}\right\rfloor$.
Let $G$ be a $2$-connected
$n$-vertex $\mathcal{C}_{\geq k}$-free graph with the maximum degree power $\sum_{v\in V(G)}d^p(v)$ where $p\geq 2$, Then, $G\in \mathcal{G}^1_{n,k}=\{W_{n,k,s}=K_{s}\vee ((n-k+s)K_1\cup K_{k-2s}): 2\leq s\leq t\}$. 
\end{thm}

For paths, we shall prove a more exact result as follows.
\begin{thm}\label{Im-co}
 Let $n\geq k\geq 4$ and $t=\left\lfloor\frac{k}{2}\right\rfloor-1$.
Let $G$ be a connected
$n$-vertex $P_{k}$-free graph with the maximum degree power $\sum_{v\in V(G)}d^p(v)$ where $p\geq 2$. Then, $G$ is  $W_{n,k-1,t}$, or $W_{n,k-1,1}$.
\end{thm}

\begin{proof}
Note that $p$-th degree power is feasible. Give the same settings as Theorem~\ref{Thm:Main2}. By Theorem~\ref{Thm:Main2}, $G$ is a member of $$\mathcal{G}^2_{n,k}=\{W_{n,k-1,s}=K_{s}\vee ((n-k+s+1)K_1\cup K_{k-2s-1}): 1\leq s\leq t\}.$$ Set $G=W_{n,k-1,s}$. Observe that $$\sum_{v\in V(G)}d^p(v)=(n-k+s+1)s^p+s(n-1)^p+(k-2s-1)(k-s-2)^p.$$ Fix $n$ and $k$. Denote $\sum_{v\in V(G)}d^p(v)$ by $D_p(s)$. We have 	\[
		\begin{array}{lll} \vspace{0.1cm}
			\frac{\mathrm{d} D_p}{\mathrm{d} s}=s^p+p(n-k+s+1)s^{p-1}+(n-1)^p-2(k-s-2)^p-p(k-2s-1)(k-s-2)^{p-1},  \\
			\frac{\mathrm{d}^{2}D_p}{\mathrm{d}s^2}=ps^{p-1}+p(p-1)(n-k+s+1)s^{p-2}+2p(k-s-2)^{p-1}+2p(k-s-2)^{p-1}\\+p(p-1)(k-2s-1)(k-s-2)^{p-2}. \\
		\end{array}
		\]\\
  Note that $\frac{\mathrm{d}^{2}D_p}{\mathrm{d}s^{2}}>0$. Thus, $D(r)$ is a strictly convex function for the domain, which concludes the result.  
\end{proof}

With the above theorem, we prove the following which tried to attack a problem of Caro and Yuster.
\begin{thm}\label{Im-coo}
Let $k\geq 4$ and $n\geq 2k$. Let $G$ be a connected $n$-vertex $P_{k}$-free graph with the maximum degree power $\sum_{v\in V(G)}d^p(v)$ where $p\geq 2$. Then, $G$ is $W_{n,k-1,t}$ where $t=\lfloor\frac{k}{2}\rfloor-1$. 
\end{thm}
\begin{proof}
Let $G_1$ be $K_{\left\lfloor\frac{k}{2}\right\rfloor-1}\vee ((n-\lceil\frac{k}{2}\rceil)K_1\cup K_{k-2\lfloor\frac{k}{2}\rfloor+1})$ and let $G_2$ be $K_{1}\vee ((n-k+2)K_1\cup K_{k-3})$. By Theorem~\ref{Im-co}, we only need to prove that the degree power of $G_1$ is more than the degree power of $G_2$ ($k\ne 4,5$). First, we assume that $k$ is even. We have that \begin{equation}\label{e1}
        \sum_{v\in V(G_1)}d^p(v)=\left(\frac{k}{2}-1\right)(n-1)^p+(n-\frac{k}{2}+1)\left(\frac{k}{2}-1\right)^p
    \end{equation} 
    and
    \begin{equation}\label{e2}
        \sum_{v\in V(G_2)}d^p(v)=n-k+2+(n-1)^p+(k-3)^{p+1}.   
    \end{equation}
    Let $D(n)= (\ref{e1})-(\ref{e2})$. If $k=4$, then we have $G_1\cong G_2$. So, we first assume that $k\geq 6$. We have
    \begin{eqnarray}
        D(n)&=&\left(\frac{k}{2}-2\right)(n-1)^p-(k-3)(k-3)^{p}+(n-\frac{k}{2}+1)\left(\frac{k}{2}-1\right)^p-n+k-2\nonumber\\
        &>& \left(\frac{k}{2}-2\right)(n-1)^p-(k-3)(k-3)^{p}\nonumber\\
        &\geq& \left(\frac{k}{2}-2\right)(2k-1)^p-(k-3)(k-3)^{p}\nonumber\\
        &=&\left(\frac{k}{2}-2\right)2^p\left(k-\frac{1}{2}\right)^p-(k-3)(k-3)^{p}\nonumber\\
        &>&(k-5)(k-3)^p\nonumber\\
        &>& 0.\nonumber
    \end{eqnarray}
 If $k$ is odd, we have that \begin{equation}\label{e3}
        \sum_{v\in V(G_1)}d^p(v)=\left(\frac{k-1}{2}-1\right)(n-1)^p+\left(n-\frac{k+1}{2}\right)\left(\frac{k-1}{2}-1\right)^p+2\left(\frac{k-1}{2}\right)^p.
    \end{equation}     
    Now, let $D(n)=(\ref{e3})-(\ref{e2})$. If $k=5$, then we have $G_1\cong G_2$. So, we assume that $k\geq 7$ is what follows. Then
    \begin{eqnarray}
D(n)&=&\left(\frac{k-1}{2}-2\right)(n-1)^p-(k-3)(k-3)^{p}+\left(n-\frac{k+1}{2}\right)\left(\frac{k-1}{2}-1\right)^p\nonumber\\
&+&2\left(\frac{k-1}{2}\right)^p-n+k-2\\
        &>& \left(\frac{k-1}{2}-2\right)(n-1)^p-(k-3)(k-3)^{p}\nonumber\\
        &\geq& \left(\frac{k-1}{2}-2\right)(2k-1)^p-(k-3)(k-3)^{p}\nonumber\\
        &=&\left(\frac{k-1}{2}-2\right)2^p\left(k-\frac{1}{2}\right)^p-(k-3)(k-3)^{p}\nonumber\\
        &>&(k-7)(k-3)^p\nonumber\\
        &>& 0.\nonumber
    \end{eqnarray}
    Thus, $ \sum_{v\in V(G_1)}d^p(v)>  \sum_{v\in V(G_2)}d^p(v)$.
\end{proof}

Recall that Caro and Yuster have proved the following result.
\begin{thm}[Caro and Yuster \cite{CY00}]
Let $k\geq 4$ and $t=\lfloor\frac{k}{2}\rfloor-1$, let $p\geq 2$ and let $n>n_0(k)$. Let $G$ be an $n$-vertex $P_{k}$-free graph with the maximum degree power $\sum_{v\in V(G)}d^p(v)$. Then, $G$ is $W_{n,k-1,t}$. Furthermore, $W_{n,k-1,t}$ is the unique extremal graph.
\end{thm}

Observe that for a connected graph, our extremal graphs, as delineated in Theorem~\ref{Im-coo}, closely align with those identified by Caro and Yuster. We determined $n_0(k)=2k$, and we posit that for all graphs, $n_0(k)$ should be $Ck^2$ for some constant $C$.

\subsection{Spectral radius}
By Theorems~\ref{Thm:Main1}, \ref{Thm:Main2}, \ref{Thm:Matching} and \ref{srf}, we conclude the following theorem immediately. 
\begin{thm}\label{sr-main}
(i) Let $n\geq k\geq 5$ and $t=\left\lfloor\frac{k-1}{2}\right\rfloor$.
If $G$ is a $2$-connected
$n$-vertex graph with $\lambda(G)>\max_{2\leq s\leq t}\{\lambda(W_{n,k,s})\}$,
then $G$ has a cycle of length at least $k$. In particular, $spex_{2-con}(n,\mathcal{C}_{\geq k})=\max_{2\leq s\leq t}\{\lambda(W_{n,k,s})\}$.\\
(ii) Let $n\geq k\geq 4$ and $t=\left\lfloor\frac{k}{2}\right\rfloor-1$.
If $G$ is a connected
$n$-vertex graph with $\lambda(G)>\max_{1\leq s\leq t}\{\lambda(W_{n,k-1,s})\},$
then $G$ has a path of order $k$. In particular, $spex_{c}(n,P_k)=\max_{1\leq s\leq t}\{\lambda(W_{n,k-1,s})\}$.\\
(iii) Let $n\geq 2k+2$. If $G$ is a connected $n$-vertex graph with $\lambda(G)>\max_{1\leq s\leq k}\{\lambda(W_{n,2k+1,s})\}$, then $G$ contains an $M_{k+1}$. In particular, $spex_{c}(n,M_{k+1})=\max_{1\leq s\leq t}\{\lambda(W_{n,2k+1,s})\}$.
\end{thm}

The following theorem gives a complete solution to a problem initially posed by Nikiforov.
\begin{thm}\label{Thm:path-spectralradius}
Let $n\geq k\geq 4$ and $t=\left\lfloor\frac{k}{2}\right\rfloor-1$.
Let $G$ be an $n$-vertex $P_{k}$-free graph with the maximum spectral radius. Then $\lambda(G)=max \{\lambda(W_{n,k-1,s}),\lambda(K_{k-1}):1\leq s\leq t\}$.
\end{thm}
\begin{proof}
    Let $G_i$ be a connected graph of order $n_i(\geq k)$ with maximum spectral radius for $i=1,2$. By Theorem \ref{sr-main}, $G_1\cong W_{n_1,k-1,s_i}$ and $G_2\cong W_{n_2,k-1,s_j}$ for some $1\leq s_i\leq t$ and $1\leq s_j\leq t$. We assume $n_1<n_2$. Observe that $W_{n_1,k-1,s_i}$ is a proper subgraph of $W_{n_2,k-1,s_i}$ and $\lambda(W_{n_2,k-1,s_j})\geq \lambda(W_{n_2,k-1,s_i})$. Thus, $\lambda(W_{n_2,k-1,s_j})> \lambda(W_{n_1,k-1,s_i})$, and it follows that $\lambda(G_1)<\lambda(G_2)$. On the other hand, $\lambda(G)\geq \lambda(K_{k-1})=k-2$ since $n\geq k$. So, if $\lambda(W_{n,k-1,s_i})\geq k-2$ for some $1\leq s_i\leq t $, then $\lambda(G)=max \{\lambda(W_{n,k-1,s})\}$ where $1\leq s\leq t$. If $\lambda(W_{n,k-1,s_i})< k-2$ for all $1\leq s_i\leq t$, we have that $\lambda(G)=k-2$ as the size of a clique in a $P_{k}$-free graph is at most $k-1$. This proves the theorem.
\end{proof}
As we mentioned before, Nikiforov \cite{N10} determined the spectral extremal graph for paths when the order of a graph is sufficiently large.
\begin{thm}[Nikiforov \cite{N10}]\label{Thm:N10}
Let $k$ be a positive integer. Let $G$ be a graph on $n\geq 2^{4k}$ vertices. 
If $\lambda(G)\geq \lambda(W_{n,k-1,t})$ where $t=\lfloor\frac{k}{2}\rfloor-1$, then $G$ contains a path of order $k$ unless $G\cong W_{n,k-1,t}$. 
\end{thm}
With the help of the following result, we improve Theorem \ref{Thm:N10} from $n\geq 2^{4k}$ to $n\geq 3k$ by a different method.
\begin{lem}[Sun-Das \cite{SD20}]\label{Lem:Sun-Das}
Let $G$ be an $n$-vertex graph. For any vertex $v\in V(G)$, if $d_G(v)\geq 1$ then
$$
\lambda^2(G)\leq \lambda^2(G-v)+2d_G(v)-1.
$$
\end{lem}

\begin{thm}
Let $k$ be a positive integer. Let $G$ be a graph on $n\geq 3k$ vertices. 
If $\lambda(G)\geq \lambda(W_{n,k-1,t})$ where $t=\lfloor\frac{k}{2}\rfloor-1$, then $G$ contains a path of order $k$ unless $G\cong W_{n,k-1,t}$. 
\end{thm}
\begin{proof}
Let $G$ be an $n$-vertex $P_{k}$-free graph with the maximum spectral radius.
By Theorem \ref{Thm:path-spectralradius}, $\lambda(G)=max \{\lambda(W_{n,k-1,s}),\lambda(K_{k-1})\}$  where $1\leq s\leq t$. Recall that $W_{n,k-1,s}$ can be partitioned into three disjoint parts $A,B,C$, such that $A$ consists of $n-k+s+1$ isolated vertices, $B$ is a clique of order $s$, and $C$ is a clique of order $k-2s-1$; moreover, $(A, B)$ is complete bipartite and $B\cup C$ is a clique of order $k-s-1$, 
in which $V(A)\cup V(B)\cup V(C)=V(G)$. Observe that after deleting $|V(C)|-1$ vertices in $C$, $W_{n,k-1,s}$ is changed into a new graph $S_{n-k+2s+2,s}$. (Without loss of generality, assume that the only one vertex in $C$ which has not been deleted is $a$.) Furthermore, by direct computation, we have 
$$
\lambda(S_{n-k+2s+2,s})\leq \frac{s-1+\sqrt{(s-1)^2+4(n-k+s+2)s}}{2}.
$$
By Lemma \ref{Lem:Sun-Das}, we obtain
\begin{align*}
&\lambda^2(W_{n,k-1,s})\leq \lambda^2(S_{n-k+2s+2,s})+2(e(B\cup C)-e(B\cup \{a\}))-(|V(C)|-1)\\
&=\left(\frac{s-1+\sqrt{(s-1)^2+4(n-k+s+2)s}}{2}\right)^2+2\left(\binom{k-s-1}{2}-\binom{s+1}{2}\right)-(k-2s-2)\\
&:=f(s).
\end{align*}
Note that when $t=\frac{k}{2}-1$ and $k$
is even, $W_{n,k-1,t}=S_{n,\frac{k}{2}-1}$; when $t=\frac{k-1}{2}-1$ and $k$ is odd, $W_{n,k-1,t}=S^+_{n,\frac{k-1}{2}-1}$, where $S^+_{n,\frac{k-1}{2}-1}$ is a graph obtained from $S_{n,\frac{k-1}{2}-1}$ by adding an extra edge in the original independent set of $n-\frac{k-1}{2}+1$ vertices in $S_{n,\frac{k-1}{2}-1}$. Thus, for the first case, we have $\lambda(W_{n,k-1,t})=\lambda(S_{n,t})$.

Next, we shall prove that $f(s)$ is increasing as $s$ is increasing when $n\geq 3k$. By computation, we have
\begin{align*}
\frac{\mathrm{d}f(s)}{\mathrm{d}s}&=\frac{(s-1)+\sqrt{(s-1)^2+4(n-k+s+2)s}}{2}\left(1+\frac{2n-2k+5s+3}{\sqrt{(s-1)^2+4(n-k+s+2)s}}\right)+4-2k\\
&>\frac{\sqrt{(s-1)^2+4(n-k+s+2)s}}{2}\left(\frac{2n-2k+5s+3}{\sqrt{(s-1)^2+4(n-k+s+2)s}}\right)+4-2k\\
&=n-3k+\frac{5s+3}{2}+4\\
&>n-3k\geq 0.
\end{align*}
Thus, when $n\geq 3k$ and $k$ is even, $\max\{\lambda(W_{n,k-1,s}):1\leq s\leq t\}=\lambda(W_{n,k-1,t})= \lambda(S_{n,t})$ as one can check that when $k$ is even and $t=\frac{k}{2}-1$, $2(e(B\cup C)-e(B\cup \{a\}))-(|V(C)|-1)=0$.
Recall that 
$$\lambda(S_{n,t})=\frac{(t-1)+\sqrt{(t-1)^2+4(n-t)t}}{2}.$$
For this case, $\lambda(S_{n,t})>k-2$
when $n\geq 3k$ and $k$ is even. 

Now we assume $k$ is odd. For this case, $t=\frac{k-3}{2}$. Recall we have $\lambda(W_{n,k-1,t})> \lambda(S_{n,\frac{k-3}{2}})$ ($S_{n,\frac{k-3}{2}}$ is obtained from $W_{n,k-1,t}$ by deleting only one edge) and $f(s)$ is monotonically increasing for $1\leq s \leq t$.  Now we claim $\lambda^2(S_{n,\frac{k-3}{2}})> f(t-1)$.  We have $$f(t-1)=\frac{\left(t-2+\sqrt{(t-2)^2+4(n-t-2)(t-1)}\right)^2}{4}+6t+3.$$

On the other hand, we obtain
$$\lambda^2(S_{n,\frac{k-3}{2}})=\frac{\left(t-1+\sqrt{(t-1)^2+4(n-t)t}\right)^2}{4}.$$
Then,
\begin{align*}
&\frac{\left(t-1+\sqrt{(t-1)^2+4(n-t)t}\right)^2}{4}-\frac{\left(t-2+\sqrt{(t-2)^2+4(n-t-2)(t-1)}\right)^2}{4}-(6t+3)\\
&>\frac{(t-1)^2+(t-1)^2+4(n-t)t}{4}-\frac{(t-2)^2+(t-2)^2+4(n-t-2)(t-1)}{4}-(6t+3)\\
&=\frac{4n+8t-14}{4}-(6t+3)\\
&=n-4t-\frac{13}{2}=n-2k-\frac{1}{2}>0.
\end{align*}
Thus, $\max\{\lambda(W_{n,k-1,s}):1\leq s\leq t\}=\lambda(W_{n,k-1,t})$. Meantime, we have
$$\lambda(S_{n,\frac{k-3}{2}})=\frac{1}{2}\cdot \left(t-1+\sqrt{(t-1)^2+4(n-t)t}\right)>2t+1$$ for $n\geq 3t+6$.
Thus, if $\lambda(G)\geq \lambda(W_{n,k-1,t})$ where $t=\lfloor\frac{k}{2}\rfloor-1$, then $G$ contains a path of order $k$ unless $G\cong W_{n,k-1,t}$. The proof is complete.
\end{proof}
\begin{thm}\label{Thm:matching-spectralradius}
Let $G$ be an $n$-vertex $M_{k+1}$-free graph with the maximum spectral radius where $n\geq 2k+1$. Then $\lambda(G)=\max \{\lambda(K_s\vee  ((n-2k+s-1)K_1\cup K_{2k-2s+1}):1\leq s\leq k\}),\lambda(K_{2k+1})\}$ where $1\leq s\leq k$.
\end{thm}
\begin{proof}
    Let $G_i$ be a connected graph of order $n_i(\geq 2k+1)$ with the maximum spectral radius for $i=1,2$. By Theorem \ref{sr-main}, $G_1\cong K_{s_i}\vee  ((n_1-2k_1+s_i-1)K_1\cup K_{2k_1-2s_i+1})$ and $G_2\cong K_{s_j}\vee  ((n_2-2k_2+s_j-1)K_1\cup K_{2k_2-2s_j+1})$ for some $1\leq s_i\leq k_1$ and $1\leq s_j\leq k_2$. We assume $n_1<n_2$ and $k_1<k_2$. Observe that $K_{s_i}\vee  ((n_1-2k_1+s_i-1)K_1\cup K_{2k_1-2s_i+1})$ is a proper subgraph of $K_{s_i}\vee  ((n_2-2k_2+s_i-1)K_1\cup K_{2k_2-2s_i+1})$ and $\lambda(K_{s_j}\vee  ((n_2-2k_2+s_j-1)K_1\cup K_{2k_2-2s_j+1}))\geq \lambda(K_{s_i}\vee  ((n_2-2k_2+s_i-1)K_1\cup K_{2k_2-2s_i+1}))$. Thus, $\lambda(K_{s_j}\vee  ((n_2-2k_2+s_j-1)K_1\cup K_{2k_2-2s_j+1}))> \lambda(K_{s_i}\vee  ((n_1-2k_1+s_i-1)K_1\cup K_{2k_1-2s_i+1}))$, and it follows that $\lambda(G_1)<\lambda(G_2)$. On the other hand, $\lambda(G)\geq \lambda(K_{2k+1})=2k$ since $n\geq 2k+1$. Note that the order of the maximum clique is $2k+1$. If $\max \{\lambda(K_s\vee  ((n-2k+s-1)K_1\cup K_{2k-2s+1})):1\leq s\leq k\}>\lambda(K_{2k+1})$, then $\lambda(G)=\max \{\lambda(K_s\vee  ((n-2k+s-1)K_1\cup K_{2k-2s+1})):1\leq s\leq k\}$; otherwise, $\lambda(G)=\lambda(K_{2k+1})$.  This proves the theorem.
\end{proof}
\subsection{Signless Laplacian spectral radius}

For the signless Laplacian spectral radius version, we also have the following theorems. (We omit the results related to matching here.)

\begin{thm}\label{slsr-mian}
(i) Let $n\geq k\geq 5$ and $t=\left\lfloor\frac{k-1}{2}\right\rfloor$.
If $G$ is a 2-connected
$n$-vertex graph with $q(G)>\max_{2\leq s\leq t}\{q(W_{n,k,s})\}$,
then $G$ has a cycle of length at least $k$. In particular, $sspex_{2-co}(n,\mathcal{C}_{\geq k})=\max_{2\leq s\leq t}\{q(W_{n,k,s})\}$.\\
(ii) Let $n\geq k\geq 4$ and let $t=\left\lfloor\frac{k}{2}\right\rfloor-1$.
If $G$ is a connected
$n$-vertex graph with $q(G)>\max_{1\leq s\leq t}\{q(W_{n,k-1,s})\},$
then $G$ has a path of order $k$. In particular, $sspex_{co}(n,P_k)=\max_{1\leq s\leq t}\{q(W_{n,k-1,s})\}.$
\end{thm}
\begin{thm}\label{Thm:path-lapspectralradius}
Let $n\geq k\geq 4$ and $t=\left\lfloor\frac{k}{2}\right\rfloor-1$.
Let $G$ be an $n$-vertex $P_{k}$-free graph with the maximum signless Laplacian spectral radius. Then $q(G)=max \{q(W_{n,k-1,s}),q(K_{k-1}): 1\leq s\leq t\}$.
\end{thm}
\begin{proof}
    Let $G_i$ be a connected graph of order $n_i(\geq k)$ with maximum spectral radius for $i=1,2$. By Theorem \ref{slsr-mian}, $G_1\cong W_{n_1,k-1,s_i}$ and $G_2\cong W_{n_2,k-1,s_j}$ for some $1\leq s_i\leq t$ and $1\leq s_j\leq t$. We assume $n_1<n_2$. Observe that $W_{n_1,k-1,s_i}$ is a proper subgraph of $W_{n_2,k-1,s_i}$ and $q(W_{n_2,k-1,s_j})\geq q(W_{n_2,k-1,s_i})$. Thus, $q(W_{n_2,k-1,s_j})> q(W_{n_1,k-1,s_i})$, and it follows that $q(G_1)<q(G_2)$. On the other hand, $q(G)\geq q(K_{k-1})=2k-4$ since $n\geq k$. So, if $q(W_{n_2,k-1,s_i})\geq 2k-4$ for some $1\leq s_i\leq t $, then $q(G)=max \{q(W_{n,k-1,s})\}$ where $1\leq s\leq t$. If $q(W_{n_2,k-1,s_i})< 2k-4$, we have that $q(G)=2k-4$ as the size of a clique in a $P_{k}$-free graph is at most $k-1$. This proves the theorem.
\end{proof}
\section{Proofs of main theorems}\label{pfsec}
Before providing proofs of our three main theorems, we need to undertake some preparatory work.

\subsection{Kelmans Operation and related lemmas}

We say that a pair of adjacent vertices $x$ and $y$ is a \emph{bad pair} if $N[x]\backslash N[y]\neq \emptyset$
and  $N[y]\backslash N[x]\neq \emptyset$.

\begin{definition}
Let $G$ be a graph. We say that $G'$ is a \emph{threshold graph} of $G$ if $G'$ is the resulting graph 
 of $G$ after doing KO for all bad pairs. 
\end{definition}
Evidently, if $G$ is connected, then the threshold graph $G'$ remains connected.
Let $S_{n,t}$ denote a graph formed by joining a clique of order $t$ with an independent set of order $n-t$.
Note that the maximum clique in $S_{n,t}$ has an order of $t+1$.
From the subsequent lemma, we can see that the threshold graph is well-defined.
\begin{lem}\label{Lem:Kelmans-structures}
Let $G$ be a connected graph and  $G'$ a threshold graph of $G$. Then the following statements hold:\\
(i) For any $xy\in E(G')$, $N[x]\subseteq N[y]$ or $N[y]\subseteq N[x]$;\\
(ii) Any two maximal cliques in $G'$ share at least one common vertex;\\
(iii) For any two maximal cliques $X$ and $Y$, there are no edges between $X-Y$ and $Y-X$.\\
(iv) Let $S$ be a maximum clique of $G'$. Then for any two maximal cliques $X$ and $Y$, we have $X\cap S\subseteq Y\cap S$ or $Y\cap S\subseteq X\cap S$.
\end{lem}

\begin{proof}
(i) Note that after each KO, the degree sequence is going up.  Then it follows from the definition of a threshold graph. (For details, see Lemma 3.1(1) in \cite{LN21}.)
  
(ii) Suppose that $X$ and $Y$ are two disjoint maximal cliques in $G$. Then pick one vertex $x$ from $X$ and one vertex $y$ from $Y$. If $xy\in E(G')$, we have $X\cap Y\ne \emptyset$ since $x,y$ is not a bad pair. So, we may assume that there is no edge between $X$ and $Y$.  As $G'$ is connected, there is a shortest path between $x$ and $y$. Observe that the path length is at most two, and say the middle vertex $a$. Now, either $\{a\}\cup X (Y)$ is a larger clique or there is an edge between $X$ and $Y$, a contradiction.
  
(iii) We prove this by contradiction. Suppose there are two maximal cliques $X$ and $Y$ such that we can find a vertex $x\in X-Y$ and a vertex $y\in Y-X$ satisfying $xy\in G'$. Then by (i), suppose, without loss of generality, $N[x]\subseteq N[y]$, we have $X\subseteq N(y)$, a contradiction to the fact that $X$ is a maximal clique.
  
For (iv), it is not hard to check by (iii).
\end{proof}
According to the definition of KO, we can conclude the following technical result.

\begin{lem}\label{Ai-Ning}
Let $G$ be a connected graph and $G'$ a threshold graph of $G$. Let $s$ and $p$ denote the orders of the maximum clique  and 
 the second maximum clique of $G'$, respectively. 
Then after a series of EKO, $G'$ can become a new graph $G''$  which is a subgraph of $K_{p-1}\vee (K_{s-p+1}\cup (n-s)K_1)$.
Furthermore, if $G''=K_{p-1}\vee (K_{s-p+1}\cup (n-s)K_1)$, then $G'=K_{p-1}\vee (K_{s-p+1}\cup (n-s)K_1)$. 

\end{lem}
\begin{proof} 
Let $X$ be a  maximum clique of $G'$ with $|X|=s$. By Lemma \ref{Lem:Kelmans-structures}(iv),
there is a maximal clique $X'$ such that for any maximal clique $Z$, $Z\cap X\subseteq X'\cap X$, we denote $X'\cap X$ by $I$. Now let $Y_1,Y_2,\ldots,Y_{\ell}$ be all maximal cliques
of $G'$ other than $X$. Notice that $V(\bigcup_{1\leq i\leq \ell}Y_i\cup X)=V(G)$ and for any $1\leq i\leq \ell$, $V(Y_i\cap X)\neq \emptyset$ and $Y_i\cap X\subseteq I$. Without loss of generality, assume that $Y_1$ is a second maximum clique with $|Y_1|=p$. 
We have $|Y_1|>|I|$.

Let $T=X\cap Y_1$ and $|T|=t$. Then $T\subseteq I$ and let $V(Y_1-T)=B_1=\{y^1_1,y^1_2,\ldots, y^1_{p-t}\}$. We choose $A_1:=\{x^1_1,x^1_2,\ldots, x^1_{p-t-1}\}$ from $X-T$ such that $V(I-T)\subseteq A_1$.  By Lemma \ref{Lem:Kelmans-structures}(iii),
there are no edges between $X-T$ and $Y_1-T$. Now, let $G^1_1=G'$ and $G^1_{i+1}=G^1_i[y^1_i,x^1_i]$ (EKO) for $1\leq i\leq p-t-1$. In $G^1_{p-t}$, $Y_1-T$ turns into an empty graph, and $G^1_{p-t}[V(X\cup Y_1)]$ is a subgraph of $K_{p-1}\vee (K_{s-p+1}\cup (p-t)K_1)$, where $G^1_{p-t}[A_1\cup V(T)]\cong K_{p-1}$, $G^1_{p-t}[V(Y_1-T)]\cong (p-t)K_1$ and $G^1_{p-t}[V(X)\backslash(A_1\cup V(T))]\cong K_{s-p+1}$. 

 Now let $G^2=G^1_{p-t}$. By the arguments above, consider $Y_2$. Let $B_2=V(Y_2-I)=\{y^2_1,y^2_2,\dots,y^2_{|Y_2-I|}\}$. Now pick $|Y_2-I|-1$ vertices from $A_1$ which are not contained in $Y_2\cap I$ to make up $A_2$.
Let $G^2_1=G^2$ and $G^2_{i+1}=G^2_i[y^2_i,x^2_i]$ (EKO) for $1\leq i\leq |Y_2-I|-1$. In $G^2_{|Y_2-I|}$, $Y_2-I$ turns into an empty graph, $G^2_{|Y_2-I|}[V(X\cup Y_1\cup Y_2)]$ is a subgraph of $K_{p-1}\vee (K_{s-p+1}\cup (|Y_1\cup Y_2-I|)K_1)$. When it comes to $Y_i$, let $G^i=G^{i-1}_{|Y_{i-1}-I|}$ and $B_i=V(Y_i-I)=\{y^i_1,y^i_2,\dots,y^i_{|Y_i-I|}\}$. Pick $|Y_i-I|-1$ vertices from $A_1$ which are not contained in $Y_i\cap I$ to make up $A_i$. Let $G^i_1=G^i$ and $G^i_{j+1}=G^i_j[y^i_j,x^i_j]$ (EKO) for $1\leq j\leq |Y_i-I|-1$. In $G^i_{|Y_i-I|}$, $Y_i-I$ turns into an empty graph, $G^i_{|Y_i-I|}[V(X\cup Y_1\cup Y_2\cup \dots\cup Y_i)]$ is a subgraph of $K_{p-1}\vee (K_{s-p+1}\cup (|Y_1\cup Y_2\cup\dots\cup Y_i-I|)K_1)$. Keep doing the process until we get $G^{\ell}$. We have that $G^l$ is a subgraph of $K_{p-1}\vee (K_{s-p+1}\cup (|Y_1\cup Y_2\cup \dots\cup Y_l-I|)K_1)=K_{p-1}\vee (K_{s-p+1}\cup (n-s)K_1)$. 
  
If we do at least one EKO, then there exists $1\leq i\leq l$ such that 
there is no edge between $y^i_1$ and $A_1\backslash V(Y_i\cap I)$. Hence $G''$ is a proper subgraph of $K_{p-1}\vee (K_{s-p+1}\cup (n-s)K_1)$. If $G''=K_{p-1}\vee (K_{s-p+1}\cup (n-s)K_1)$, then by the arguments above, we don't do any EKO on $G'$, and so $G'=K_{p-1}\vee (K_{s-p+1}\cup (n-s)K_1)$. 
\end{proof}

Note that $K_{p-1}\vee (K_{s-p+1}\cup (n-s)K_1)$ is a subgraph of $S_{n,s-1}$. The following lemma can be deduced directly.

\begin{lem}\label{koless}
Let $G$ be a connected graph and $G'$ a threshold graph of $G$. Let $X$ be a maximum clique of $G'$ with $|X|=s$. Then $G'=S_{n,s-1}$ or after a series of EKO, $G'$ becomes a proper subgraph of $S_{n,s-1}$. 
\end{lem}
To substantiate some of our main results, we may require the assistance of the following lemmas.
\begin{lem}[Gao and Hou \cite{GH19}]\label{gao2019}
Let $G$ be a graph and $uv\in E(G)$, and $G':=G[u\rightarrow v]$. Then $c(G')\leq c(G)$.
\end{lem}
\begin{lem}\label{A-N2}
Let $G$ be a graph and $uv\in E(G)$, and $G':=G[u\rightarrow v]$. If there is a path of order $k$ in $G'$, then there is a path of order at least $k$ in $G$.
\end{lem}
\begin{proof}
    Let $P'=x_1x_2\dots x_k$ be a path of order $k$ in $G'$. If $P'$ does not contain an edge $va$ where $a\in N_G(u)\backslash N_G(v)$, then $P'$ is also a path in $G$, we are done. So, we may assume $P'$ contains such an edge $va$. If $u\notin P'$, then we can replace $va$ with $au$ and $uv$ of $P'$ in $G$ to get a longer path, and we are done. If $P'$ does not contain an edge $vb$ where $b\in N_G(v)\backslash N_G(u)$, we can easily swap $u$ and $v$ in $P'$ to get a new path $P$ in $G$, observe that $|P|=|P'|$, we are done. So, we assume such an edge $vb$ exists in $P'$. Now we complete this proof by considering the following two cases:\\
    \noindent{Case 1: $d_{P'}(u)=1$.}\\
    Let $P=P'-\{av\}+\{au\}$, then $P$ is a path of the same order as $P'$ in $G$.\\
    \noindent{Case 2: $d_{P'}(u)=2$.}\\
    Let $c$ and $d$ be the two neighbors of $u$ in $P'$, then $u,v \in N(u)\cap N(v)$. Without loss of generality, we assume $u=x_i$ and $v=x_j$ such that $i<j$ and $c=x_{i+1}$. Now let $P=P'-\{av, ud\}+\{au, vd\}$, then $P$ is a path of the same order as $P'$ in $G$.
\end{proof}
\begin{lem}[Li and Ning \cite{LN21}]\label{LiNing}
Let $G$ be a graph, and let $x,y,u$ be distinct vertices of $G$, and $v\in N(u)$ (possibly $v\in \{x,y\}$). Let $G':=G[u\to v].$ If $G'$ has an $(x,y)$-path of length at least $k$, then so does $G$.
\end{lem}

\subsection{Kopylov's operation and related lemmas}
Our proofs need Kopylov's operation along with some lemmas presented below.

\begin{definition}
[$\alpha$-disintegration of a graph \cite{K77}]
Let $G$ be a graph and $\alpha$ be a natural number. Delete all
vertices of degree at most $\alpha$ from $G$; for the resulting graph $G'$,
we again delete all vertices of degree at most  $ \alpha$ from $G'$. We keep
running this process until finally get a graph, denoted by $H(G;\alpha)$, such
that all vertices are of degree larger than $\alpha$.
\end{definition}

\begin{lem}[Kopylov \cite{K77}]\label{Lem_Kopylov}
Let $G$ be a 2-connected $n$-vertex graph with a path $P$ of $m$
edges with endpoints $x$ and $y$. For $v\in V(G)$, let $d_P(v)=|N(v)\cap V (P)|$. Then $G$ contains
a cycle of length at least $\min\{m+1,d_P(x)+d_P(y)\}$.
\end{lem}

\begin{lem}\label{Ning-Ai}
Let $\Gamma$ be a connected $n$-vertex graph with two vertex-disjoint paths, say $F=\{P_1,P_2\}$, in which
$v(P_1)+v(P_2)=p$ and $x,y$ are end-vertices of $P_1,P_2$, respectively.
For $v\in V(G)$, let $d_F(v)=|N(v)\cap (V(P_1)\cup V(P_2))|$. Then $G$ contains
a path of order at least $\min\{p,d_F(x)+d_F(y)+1\}$.
\end{lem}
\begin{proof}
Add a new vertex $z$ and let $G:=\Gamma\vee \{z\}$. Since $\Gamma$ is connected, $G$ is $2$-connected.
Let $x'$ be the other end-vertex of $P_1$ and $y'$ the other end-vertex of $P_2$.
Let $P:=P_1x'zy'P_2$. Then $P$ is a path of order $p+1$. Moreover, $d_{P}(x)=|N_{G}(x)\cap V(P)|=d_F(x)+1$
and $d_P(y)=d_F(y)+1$. By Lemma \ref{Lem_Kopylov}, there is a cycle of length at least $\min\{p+1,d_F(x)+d_F(y)+2\}$,
say $C$. If $C$ contains the vertex $z$, then $P'=C-\{z\}$ is a path in $\Gamma$ with order at least $\min\{p,d_F(x)+d_F(y)+1\}$.
If $C$ contains no $z$, then deleting any edge of $C$ gives a path in $\Gamma$ with order
at least $\min\{p+1,d_F(x)+d_F(y)+2\}$. This proves the lemma.
\end{proof}
\subsection{Proofs}
Now, we are ready to provide the proofs. 
\begin{proof}[Proof of Theorem \ref{Thm:Main1}]
We prove the theorem by contradiction. Let $G$ be an $n$-vertex $2$-connected graph containing no $\mathcal{C}_{\geq k}$ with the maximum $\mathcal{P}(G)$ but $G\notin \mathcal{G}^1_{n,k}$. 
Since $\mathcal{P}(G)$ is maximum, by Property (II), $G$ has the maximal number of edges. So, adding any new edge (i.e., joining any two non-adjacent vertices) will increase the value of $\mathcal{P}(G)$. Thus, $G$ is edge-maximal,
 and adding any new edge creates a cycle of length at least $k$ in the
resulting graph. We state it in another form.
\begin{claim}
For any two non-adjacent vertices $x,y\in V(G)$, there is a path with $x,y$ as two
end-vertices in $G$ of order at least $k$.
\end{claim}
For any two adjacent vertices $x,y\in V(G)$, if neither $N_G(x)\subset N_G(y)$ nor $N_G(y)\subset N_G(x)$, we use the KO to get a new graph $G_{xy}=G[x\rightarrow y]$ or $G_{yx}=G[y\rightarrow x]$. By Property (I),
we have $\mathcal{P}(G_{xy})\geq \mathcal{P}(G)$
and $\mathcal{P}(G_{yx})\geq \mathcal{P}(G)$. 

After a series of KO, the procedure will stop and result in a threshold graph, denoted by $\Gamma$. In the following,
let $G:=G_0, G_1, G_2,\dots, G_h:=\Gamma$ be a sequence of graphs, such that $G_{i+1}=G_i[u_i,v_i]$, where $u_iv_i\in E(G_i)$
and $G_h$ is the threshold graph of $G$.
Hence $\mathcal{P}(\Gamma)\geq \mathcal{P}(G)$.
\begin{claim}\label{nocut}
    $G\cong\Gamma$.
\end{claim}
\pf As $n\geq k$, we assert there are at least two maximal cliques in $\Gamma$. Let $X_i$ and $X_j$ be any two maximal cliques. For any vertex $x\in X_i\backslash X_j$ and $y\in X_j\backslash X_i$, we have $xy\notin E(\Gamma)$. Recall that $G_{i+1}=G_i[u_i,v_i]$, where $u_iv_i\in E(G)$
and $G_h$ is the threshold graph of $G$. 

If there is a $j\in[h-1]$ such that $\{u_j,v_j\}=\{x,y\}$, then pick
the largest $j$ such that $G_{j+1}=G_j[x,y]$ or $G_{j+1}=G_j[y,x]$. Set $G_{j+1}$ as a new $G_0$, then there is a path $P_{xy}$ between $x$ and $y$ of length at least $k-1$ in $G_{j+1}$ by Lemma \ref{LiNing}. Note that $xy\in E(G_{j+1})$, hence $P_{xy}+xy$ is a cycle of length at least $k$ in $G_{j+1}$. Now applying Lemma~\ref{gao2019}, we have that there is a cycle of length at least $k$ in $G$ as well, a contradiction.
 
Now for any $i$, $|\{u_i,v_i\}\cap\{x,y\}|\leq 1$, and so there is also an $(x,y)$-path of length at least $k-1$ in $G$ by Lemma \ref{LiNing}. If $xy\in E(G)$, we are done since $P_{xy}+xy$ is a cycle of length at least $k$, a contradiction. So, we may assume $xy\notin E(G)$. Then we have $G=\Gamma$ as Kelmans Operation keeps the number of edges.
\qee

\begin{claim}
Let $H=H(\Gamma;t)$. Then $H$ is not empty.
\end{claim}
\pf
Suppose to the contrary that $H=\emptyset$. 
Choose a maximum clique in $\Gamma$, and denote it by $X$.   

Let $|X|=s$ and $|Y|=n-s$. Recall that $\Gamma[X]=K_s$. Since $H=\emptyset$,
$s\leq t+1$. Indeed, if $s\geq t+2$, then there is a $K_{t+2}$-clique in $\Gamma$.
After all $t$-disintegrations of $\Gamma$, the $K_{t+2}$-clique is still a $K_{t+2}$-clique
in $H$, contradicting the fact that $H=\emptyset$. This shows that $s\leq t+1$.
Assume after applying a series of EKO to $\Gamma$, $\Gamma$ becomes a proper subgraph of $S_{n, s-1}$. Notice that the largest cycle in $S_{n, s-1}$ is of length at most $2s-2$ and $2s-2\leq 2t\leq k-1$. Then there is a contradiction to the fact that $\mathcal{P}(G)$ is maximum since $S_{n, s-1}$ is $\mathcal{C}_{\geq k}$-free graph.  Thus, by Claim~\ref{nocut} and Lemma~\ref{koless}, $G=S_{n, s-1}=W_{n,k-1,s-1}$, which is a contradiction to the assumption that $G\notin\mathcal{G}^1_{n,k}$. \qee

\begin{claim}\label{Claim:H-clique}
$H$ is a clique.
\end{claim}
\pf
We shall show that $H$ is a clique. Suppose $x,y \in V(H)$ are not adjacent in $H$. Then $x$ and $y$ are not adjacent in $\Gamma$ either. As $G$ is edge-maximal, $|E(G)|=|E(\Gamma)|$ and $\mathcal{P}(\Gamma+xy)>\mathcal{P}(\Gamma)$, we know $\Gamma+xy$ contains a cycle of length at least $k$.
Thus, there is an $(x,y)$-path of length at least $k-1$
in $\Gamma$.

By Claim~\ref{nocut}, $\Gamma$ is $2$-connected.
Without loss of generality, we choose $x,y\in V(H)$ with $xy\notin E(H)$ such that the length of a longest $(x,y)$-path is the maximum among all possible pairs $x,y$, say $P$, with these two vertices as end-vertices in $\Gamma$.
We claim $N_H(x) \subseteq V(P)$ and $N_H(y) \subseteq V(P)$. Suppose $z \in N_H(x) $ and $z \not \in V(P)$. If $yz \in E(G)$, then we get a cycle of order at least $k+1$, a contradiction. If $y,z$ are nonadjacent, then we get two new vertices $y,z$ such that there is a longer path $yz+P$ with these two vertices as end-vertices, also a contradiction. The same argument also holds for $y$. 
  
By Lemma~\ref{Lem_Kopylov}, we get a
cycle of length at least $\min \{k-1+1,d_H(x)+d_H(y) \}\geq\min \{k,2(t+1)\} \geq k$ in $\Gamma$. According to Lemma~\ref{gao2019}, there is a cycle of length at least $k$ in $G$ as well, a contradiction.
 This proves the claim.
\qee

\begin{claim}\label{cl3}
$H$ is a clique with the maximum size in $\Gamma$.
\end{claim}
\pf
Suppose that there exists another clique, say $H'$ in $\Gamma$, such that $|H'|>|H|$. Then for any vertex $v\in V(H')$, $d_{H'}(v)\geq |H'|-1\geq |H|\geq t+2$.
As $H'$ is a clique in $\Gamma$, any vertex in $H'$ cannot be deleted in $H(\Gamma;t)$, and hence $H'\subseteq H$,
contradicting the fact that $|H'|>|H|$ . This proves the claim.
\qee

\begin{claim}\label{claim-inequality}
Let $r=|V(H)|$. Then $t+2\leq r\leq k-2$, and so $2\leq k-r\leq t.$
\end{claim}
\pf
As $H=H(\Gamma;t)$ is a clique, $r \geq t+2$. We claim that $r\leq k-2$. If $r\geq k$, then there is a cycle of length at least $k$, a contradiction.
Thus, we may assume that $r=k-1$. By Claim~\ref{nocut}, $\Gamma$ is $2$-connected, so there is a cycle of length at least $k$ in $\Gamma$, as for each vertex not in $H$, say $a$, there are two disjoint paths between $a$ and $H$. 
Then $c(G)\geq k$ by Lemma~\ref{gao2019}, a contradiction. 
So, $k-r\leq k-t-2=k-\left\lfloor\frac{k-1}{2}\right\rfloor-2\leq \left\lfloor\frac{k-1}{2}\right\rfloor=t$.
This proves the claim.
\qee

\begin{claim}\label{claim-convex}
Let $H'=H(\Gamma;k-r)$. Then $H \neq H'$.
\end{claim}
\pf
Suppose $H=H'$. Note that each vertex from $V(G)\setminus V(H')$ has degree at most $k-r$. So the size of the second largest maximal clique is at most $k-r+1$. By Lemma~\ref{Ai-Ning} and Claim 2, after applying a series of EKO to $\Gamma$, $\Gamma$ becomes a proper subgraph of $K_{k-r}\vee ((n-r)K_1\cup K_{2r-k})\in\mathcal{G}^1_{n,k}$ or $G=\Gamma=K_{k-r}\vee ((n-r)K_1\cup K_{2r-k})$, then
$\mathcal{P}(\Gamma)< \mathcal{P}(K_{k-r}\vee ((n-r)K_1\cup K_{2r-k})),$ or $G\in\mathcal{G}^1_{n,k}$,
a contradiction.  This proves the claim.
\qee

\begin{claim}
$G$ contains a cycle of length at least $k$.
\end{claim}

Thus $V(H)  \subset V(H')$ and $H \neq H'$. We select $x \in V(H) $ and $y \in V(H') \setminus V(H)$ such that $x$ and $y$ are nonadjacent and the longest path between them in $\Gamma$ contains the largest number of edges among all such pairs. Let $P$ be a longest path between $x$ and $y$, we can assert that the length of $P$ is at least $k-1$ since otherwise $\mathcal{P}(\Gamma+xy)>\mathcal{P}(\Gamma)$. Now, we claim $N_H(x) \subseteq V(P)$ and $N_{H'}(y) \subseteq V(P)$. By Claim~\ref{nocut}, $\Gamma$ is $2$-connected. Then Lemma~\ref{Lem_Kopylov} implies that there is a cycle in $\Gamma$ with length at least $\min \{k, d_P(x)+d_P(y)\} \geq \min \{k,r-1+k-r+1\}=k$. By Lemma~\ref{gao2019}, there is a cycle of length at least $k$ in $G$. We get a contradiction.
\end{proof}

A basic fact in graph theory is the following: A graph $G$ contains a path of order $k$ if and only if $G\vee K_1$ contains a cycle of order at least $k+1$. However, we cannot deduce Theorem~\ref{Thm:Main2} from Theorem \ref{Thm:Main1} with the aid of this idea. Some further discussions can be found in our last section.

\begin{proof}[Proof of Theorem \ref{Thm:Main2}]
We prove the theorem by contradiction. Let $G$ be an $n$-vertex connected graph containing no $P_{k}$ with the maximum $\mathcal{P}(G)$ but $G\notin \mathcal{G}^2_{n,k}$. 
Since $\mathcal{P}(G)$ is maximum, by Property (II), $G$ has the maximal number of edges. Adding any new edge (i.e., joining any two non-adjacent vertices) will increase the value of $\mathcal{P}(G)$. Thus, $G$ is edge-maximal. We have the following.
\setcounter{claim}{0}
\begin{claim}
For any two non-adjacent vertices $x,y\in V(G)$, there are two vertex-disjoint paths $P_1=P(x)$ and $P_2=P(y)$,
in which one has $x$ as an end-vertex and the other has $y$ as an end-vertex, and $v(P_1)+v(P_2)\geq k$.
\end{claim}

For any two adjacent vertices $x,y\in V(G)$, if neither $N_G(x)\subset N_G(y)$ nor $N_G(y)\subset N_G(x)$, we use the KO to get a new graph $G_{xy}=G[x\rightarrow y]$ or $G_{yx}=G[y\rightarrow x]$. By Property (I),
we have $\mathcal{P}(G_{xy})\geq \mathcal{P}(G)$
and $\mathcal{P}(G_{yx})\geq \mathcal{P}(G)$. 

After a series of KO, the procedure will stop and result in a threshold graph, denoted by $\Gamma$. In the following,
let $G:=G_0, G_1, G_2,\dots, G_h:=\Gamma$ be a sequence of graphs, such that $G_{i+1}=G_i[u_i,v_i]$, where $u_iv_i\in E(G)$
and $G_h$ is the threshold graph of $G$.
Hence $\mathcal{P}(\Gamma)\geq \mathcal{P}(G)$.
\begin{claim}\label{2.2}
If  $\Gamma\in\mathcal{G}^2_{n,k}$, then $G\cong\Gamma$.
\end{claim}
\pf  Suppose $\Gamma=W_{n,k-1,s}=K_{s}\vee ((n-k+s+1)K_1\cup K_{k-2s-1})$, where $1\leq s\leq t$. We partition $V(W_{n,k-1,s})$   into three disjoint parts $A,B,C$, such that $A$ consists of $n-k+s+1$ isolated vertices, $B$ is a clique of order $s$, and $C$ is a clique of order $k-2s-1$; moreover, $(A, B)$ is complete bipartite and $B\cup C$ is a clique of order $k-s-1$. 
  
Now, we consider $G_{h-1}$. Recall that $\Gamma=G_{h-1}[u_{h-1}, v_{h-1}]$. For simplicity, we denote $u_{h-1}$ by $u$ and $v_{h-1}$ by $v$. 
By the definition of KO, we have $N_\Gamma(u)\subset N_\Gamma(v)$ and $uv\in E(\Gamma)$.  Notice that $N_{\Gamma}(u)=N_{\Gamma}(v)$ when $u,v\in B$ or $u,v\in C$. Thus $v\in B$ and $u\in A\cup C$. Suppose $G_{h-1}\not\cong \Gamma$. Then 
$N_{G_{h-1}}(u)\cap (A\cup C)\neq\emptyset$,  
$N_{G_{h-1}}(v)\cap (A\cup C)\neq\emptyset$ and $A\cup C\subset N_{G_{h-1}}(u)\cup N_{G_{h-1}}(v)$.

Suppose that $v\in B$ and $u\in A$. We have 
 $B-v \subseteq N_{G_{h-1}}(u)\cap N_{G_{h-1}}(v)$. Suppose $N_{G_{h-1}}(v)\cap C=\emptyset$ or $N_{G_{h-1}}(u)\cap A=\emptyset$.
 Let $a\in N_{G_{h-1}}(v)\cap A$ and 
$b\in N_{G_{h-1}}(u)\cap C$. Let $G_{u,v,a}$ be the graph obtained by deleting the three vertices $u,v,a$ from $G_{h-1}$. Since $G_{u,v,a}\cong K_{s-1}\vee ((n-k+s-1)K_1\cup K_{k-2s-1})$, there is a path $P$ of order $k-3$ ending at $b$ in $G_{u,v,a}$.
We can extend $P$ from $b$ by adding $bu$, $uv$, and $va$ such that the order of $P$ now is $k$, a contradiction. The case $N_{G_{h-1}}(u)\cap C=\emptyset$ or $N_{G_{h-1}}(v)\cap A=\emptyset$ is similar, now consider $N_{G_{h-1}}(x)\cap A\neq\emptyset$ and  
 $N_{G_{h-1}}(x)\cap C\neq\emptyset$, where $x\in\{u,v\}$. Without loss of generality assume $|N_{G_{h-1}}(u)\cap A|\geq2$.
 Let $a\in N_{G_{h-1}}(v)\cap A$ and 
$b\in N_{G_{h-1}}(v)\cap C$. Let $G_{v}$ be the graph obtained by deleting $v$ from $G_{h-1}$. Notice that there are two disjoint paths $P_1$ and $P_2$ in 
$G_{v}$, where
 $P_1$ is a path of order $2s$ ending at $a$ and $P_2$ is a path of order $k-2s-1$ ending at $b$.
We can extend $P_1$  and $P_2$ 
 from $a$ and $b$, respectively, by adding $av$ and $vb$ such that the order of the new path is $k$, a contradiction.
 
Suppose that $v\in B$ and $u\in C$.
We have 
$(B\cup C)-\{u,v\} \subseteq N_{G_{h-1}}(u)\cap N_{G_{h-1}}(v)$.
Since $G_{h-1}\not\cong \Gamma$, $N_{G_{h-1}}(u)\cap A\neq\emptyset$ and  
$N_{G_{h-1}}(v)\cap A\neq\emptyset$.
Without loss of generality assume $|N_{G_{h-1}}(u)\cap A|\geq2$.
 Let $a\in N_{G_{h-1}}(v)\cap A$ and 
$b\in N_{G_{h-1}}(u)\cap A \backslash \{a\}$.
Let $G_{v,b}$ be the graph obtained by deleting the two vertices $v,b$ from $G_{h-1}$. Notice that  there are two disjoint paths $P_1$ and $P_2$ in 
$G_{v,b}$, where
 $P_1$ is a path of order $2s-1$ ending at $a$ and $P_2$ is a path of order $k-2s-1$ ending at $u\in C$.
We can extend $P_1$ and $P_2$ 
 from $a$ and $u,y$ by adding $av$, $ub$ such that the order of the new path is $k$, a contradiction.

So $G_{h-1}\cong \Gamma$. Hence  $G\cong G_1\cong \cdots \cong G_{h-1}\cong \Gamma$.
\qee

In the following, let $H=H(\Gamma;t)$.
\begin{claim}
$H$ is not empty.
\end{claim}
\pf
Suppose to the contrary that $H=\emptyset$. 
Choose a maximum clique in $\Gamma$, and denote it by $X$.   

Let $|X|=s$ and $|Y|=n-s$. Recall that $\Gamma[X]=K_s$. Since $H=\emptyset$,
$s\leq t+1$. Indeed, if $s\geq t+2$, then there is a $K_{t+2}$-clique in $\Gamma$.
After all $t$-disintegrations of $\Gamma$, the $K_{t+2}$-clique is still a $K_{t+2}$-clique
in $H$, contradicting the fact that $H=\emptyset$. This proves that $s\leq t+1$. 
By Lemma~\ref{koless}, we have $\Gamma=S_{n, s-1}$ or after applying a series of EKO to $\Gamma$, $\Gamma$ becomes a graph $\Gamma'$ which is a proper subgraph of $S_{n, s-1}$.
Note that $S_{n, t}$ is $P_{k}$-free because the longest path is of order at most $2t+1$ and $2t+1\leq k-1$. 
Since $S_{n, s-1}(s\leq t)$ and $\Gamma'$ is a proper subgraph of  $S_{n, t}$, we have  $\Gamma=S_{n, t}$; otherwise it contradicts the fact that $\mathcal{P}(G)$ is maximum. If $\Gamma=S_{n,t}$ and $k$ is odd, then $\Gamma$ is a proper subgraph of $S^+_{n,t}$, it contradicts that $\mathcal{P}(G)$ is maximum. If $\Gamma=S_{n,t}$ and $k$ is even, by Claim~\ref{2.2}, we have $G\cong\Gamma$, then a contradiction to the assumption that $G\notin \mathcal{G}^2_{n,k}$.
\qee  

\begin{claim}\label{Claim:H-clique-2}
$H$ is a clique.
\end{claim}
\pf
We shall show that $H$ is a clique. Suppose $x,y \in V(H)$ are not adjacent in $H$.  By Claim 1, there are an $x$-path and a $y$-path such that the sum of their orders is at least $k$.  We choose such $x,y\in V(H)$ that an $x$-path $P_1$ and a $y$-path $P_2$ satisfies that $P_1$ and $P_2$ are vertex-disjoint and $|P_1|+|P_2|$
is maximum in $H$. We claim $N_H(x) \subset V(P_1\cup P_2)$ and $N_H(y) \subset V(P_1\cup P_2)$. Suppose $z \in N_H(x) $ and $z \not \in V(P_1\cup P_2)$. If $\{y,z\} \in E(G)$, then we get a path of order at least $k+1$, a contradiction. If $y$ and $z$ are nonadjacent, it is a contradiction to the choice of $P_1$. The same argument also holds for $y$. By Lemma~\ref{Ning-Ai}, we get a path of order at least $\min \{k,d_H(x)+d_H(y)+1 \}\geq \min \{k,2(t+1)+1\} \geq k$, a contradiction. This proves Claim \ref{Claim:H-clique-2}.
\qee
  
By a similar argument to the one for the cycle above, we can get the following claim. 
\begin{claim}
    $H$ is a clique with the maximum size in $\Gamma$.
\end{claim}

\begin{claim}\label{claim-inequality-2}
Let $r-1=|V(H)|$. Then $t+3\leq r \leq k-1$. So, $1\leq k-r\leq t$.
\end{claim}
\pf
As $H=H(\Gamma;t)$ is a clique, $r \geq t+3$. We claim that $r\leq k-1$.
Suppose that $r\geq k$. There is a path of order $k-1$ in $H$. Furthermore, there is a $P_k$ in $\Gamma$ since $\Gamma$ is connected, then there is a $P_k$ in $G$, a contradiction. So, $k-r\leq k-t-3=k-\left\lfloor\frac{k}{2}\right\rfloor-2\leq \left\lfloor\frac{k}{2}\right\rfloor-1=t$.
This proves Claim \ref{claim-inequality-2}.
\qee

\begin{claim}\label{claim-convex-2}
Let $H'=H(\Gamma;k-r)$. Then $H \neq H'$.
\end{claim}
\pf
Suppose $H=H'$. Notice that each vertex from $V(G)\setminus V(H')$ has degree at most $k-r$ in $\Gamma$. So the size of the second largest maximal clique is at most $k-r+1$. By Lemma~\ref{Ai-Ning} and Claim~\ref{2.2}, $\Gamma$ is a proper subgraph of $K_{k-r}\vee ((n-r+1)K_1\cup K_{2r-k-1})$, then
$\mathcal{P}(\Gamma)< \mathcal{P}(K_{k-r}\vee ((n-r+1)K_1\cup K_{2r-k-1})),$
a contradiction.  This proves the claim.
\qee

Next, we show that
$G$ contains a path of order at least $k$. By Claim \ref{claim-convex-2}, $V(H)  \subset V(H')$ and $H \neq H'$. We select $x \in V(H) $ and $y \in V(H') \setminus V(H)$ such that $x$ and $y$ are nonadjacent. By Claim 1, there are an $x$-path and a $y$-path such that the sum of their orders is at least $k$.  We choose such $x,y\in V(H)$ that an $x$-path $P_1$ and a $y$-path $P_2$ satisfies that $P_1$ and $P_2$ are vertex-disjoint and $|P_1|+|P_2|$
is maximum in $H$. We claim $N_H(x) \subset V(P_1\cup P_2)$ and $N_{H'}(y) \subset V(P_1\cup P_2)$ by a similar discussion in Claim $3$. By Lemma~\ref{Ning-Ai}, we get a path of order at least $\min \{k, d_H(x)+d_{H'}(y)\} \geq \min \{k,r-2+k-r+2\}=k$, a contradiction. This proves the theorem.
\end{proof}

To prove Theorem \ref{Thm:Matching}, we need the following lemma.
\begin{lem}[Bondy-Chv\'atal \cite{BC75}]\label{bc}
Let $G$ be a graph on $n$ vertices. For any two nonadjacent vertices $u, v\in V(G)$, if whenever
$\mu(G+uv)=k+1$ and $d_G(u)+d_G(v)\geq 2k+1$, then $\mu(G)=k+1$.
\end{lem}
Though the following proof looks similar to the above one, we give the details as there are several differences somewhere that are important. 
\begin{proof}[Proof of Theorem \ref{Thm:Matching}]
The case $n=2k+1$ is trivial, so we only consider the case $n\geq 2k+2$. We prove the theorem by contradiction. Let $G$ be an $n$-vertex connected graph containing no $M_{k+1}$ with the maximum $\mathcal{P}(G)$ but $G\notin \mathcal{G}^3_{n,k}$. 
Since $\mathcal{P}(G)$ is maximum, by Property (II), $G$ has the maximal number of edges. Adding any new edge (i.e., joining any two non-adjacent vertices) will increase the value of $\mathcal{P}(G)$. Thus, $G$ is edge-maximal. We have the following.

\setcounter{claim}{0}
\begin{claim}
For any two non-adjacent vertices $x,y\in V(G)$, $G+xy$ has a matching of size $k+1$,
i.e., $\mu(G+xy)=k+1$.
\end{claim}
For any two adjacent vertices $x,y\in V(G)$, if neither  $N_G(x)\subset N_G(y)$ nor $N_G(y)\subset N_G(x)$, we use KO to get a new graph $G_{xy}=G[x\rightarrow y]$ or $G_{yx}=G[y\rightarrow x]$. By Property (I),
we have $\mathcal{P}(G_{xy})\geq \mathcal{P}(G)$
and $\mathcal{P}(G_{yx})\geq \mathcal{P}(G)$. 
After a series of Kelmans Operations, the procedure will stop and result in a threshold graph, denoted by $\Gamma$. In the following,
let $G:=G_0, G_1, G_2,\dots, G_h:=\Gamma$ be a sequence of graphs, such that $G_{i+1}=G_i[u_i,v_i]$, where $u_iv_i\in E(G)$
and $G_h$ is the threshold graph of $G$.
Hence $\mathcal{P}(\Gamma)\geq \mathcal{P}(G)$.

\begin{claim}\label{3.2}
    If  $\Gamma\in\mathcal{G}^3_{n,k}$, then $G\cong\Gamma$.
\end{claim}
\pf  Suppose $\Gamma=W_{n,2k+1,s}=K_{s}\vee ((n-2k+s-1)K_1\cup K_{2k-2s+1})$, where $1\leq s\leq k$. We partition $V(W_{n,2k+1,s})$  into three disjoint parts $A,B,C$, such that $A$ consists of $n-2k+s-1$ isolated vertices, $B$ is a clique of order $s$, and $C$ is a clique of order $2k-2s+1$; moreover, $(A, B)$ is complete bipartite and $B\cup C$ is a clique of order $2k-s+1$. 
  
Now, we consider $G_{h-1}$. Recall that $\Gamma=G_{h-1}[u_{h-1}, v_{h-1}]$. For simplicity, we denote $u_{h-1}$ by $u$ and $v_{h-1}$ by $v$. 
By the definition of KO, we have $N_\Gamma(u)\subset N_\Gamma(v)$ and $uv\in E(\Gamma)$.  Notice that $N_{\Gamma}(u)=N_{\Gamma}(v)$ when $u,v\in B$ or $u,v\in C$. Thus $v\in B$ and $u\in A\cup C$. Suppose $G_{h-1}\not\cong \Gamma$. Then 
$N_{G_{h-1}}(u)\cap (A\cup C)\neq\emptyset$,  
$N_{G_{h-1}}(v)\cap (A\cup C)\neq\emptyset$ and $A\cup C\subset N_{G_{h-1}}(u)\cup N_{G_{h-1}}(v)$.

Suppose that $v\in B$ and $u\in A$. We have 
 $B-v \subseteq N_{G_{h-1}}(u)\cap N_{G_{h-1}}(v)$. Suppose $N_{G_{h-1}}(v)\cap C=\emptyset$ or $N_{G_{h-1}}(u)\cap A=\emptyset$.
 Let $a\in N_{G_{h-1}}(v)\cap A$ and 
$b\in N_{G_{h-1}}(u)\cap C$. Let $G_{u,v,a,b}$ be the graph obtained by deleting the four vertices $u,v,a,b$ from $G_{h-1}$. Since $G_{u,v,a,b}\cong K_{s-1}\vee ((n-2k+s-3)K_1\cup K_{2k-2s})$, there is an $M_{k-1}$ in $G_{u,v,a,b}$ as $n-2k+s-3-(s-1)=n-2k-2\geq 0$.
We can extend $M_{k-1}$ to $M_{k+1}$ by adding $bu$ and $va$, a contradiction. The case $N_{G_{h-1}}(u)\cap C=\emptyset$ or $N_{G_{h-1}}(v)\cap A=\emptyset$ is similar, now consider $N_{G_{h-1}}(x)\cap A\neq\emptyset$ and  
 $N_{G_{h-1}}(x)\cap C\neq\emptyset$, where $x\in\{u,v\}$.
 Let $a\in N_{G_{h-1}}(u)\cap A$ and 
$b\in N_{G_{h-1}}(v)\cap C$. Let $G_{a,b,u,v}$ be the graph obtained by deleting $a,b,u,v$ from $G_{h-1}$. Note that there is an $M_{k-1}$ in $G_{a,b,u,v}$. We can extend $M_{k-1}$ to $M_{k+1}$ by adding $au$ and $bv$, a contradiction.
 
Suppose that $v\in B$ and $u\in C$.
We have 
$(B\cup C)-\{u,v\} \subseteq N_{G_{h-1}}(u)\cap N_{G_{h-1}}(v)$.
Since $G_{h-1}\not\cong \Gamma$, $N_{G_{h-1}}(u)\cap A\neq\emptyset$ and  
$N_{G_{h-1}}(v)\cap A\neq\emptyset$.
Without loss of generality assume $|N_{G_{h-1}}(u)\cap A|\geq2$.
 Let $a\in N_{G_{h-1}}(v)\cap A$ and 
$b\in N_{G_{h-1}}(u)\cap (A \backslash \{a\})$.
Let $G_{b,u}$ be the graph obtained by deleting $b,u$ from $G_{h-1}$. There is an $M_k$ in $G_{h-1}$.
We can extend $M_k$ to $M_{k+1}$ by adding $bu$, a contradiction.

So $G_{h-1}\cong \Gamma$. Hence  $G\cong G_1\cong \cdots \cong G_{h-1}\cong \Gamma$.
\qee

Let $H=H(\Gamma;k)$.
\begin{claim}
$H$ is not empty.
\end{claim}
\pf
Suppose to the contrary that $H=\emptyset$. 
Choose a maximum clique in $\Gamma$, and denote it by $X$.   

Let $|X|=s$ and $|Y|=n-s$. Recall that $\Gamma[X]=K_s$. Since $H=\emptyset$,
$s\leq k+1$. Indeed, if $s\geq k+2$, then there is a $K_{k+2}$-clique in $\Gamma$.
After all $k$-disintegrations of $\Gamma$, the $K_{k+2}$-clique is still a $K_{k+2}$-clique in $H$, contradicting the fact that $H=\emptyset$. This proves that $s\leq k+1$.

 By Lemma~\ref{koless}, we have $\Gamma=S_{n, s-1}$ or after applying a series of EKO to $\Gamma$, $\Gamma$ becomes a graph $\Gamma'$ which is a proper subgraph of $S_{n, s-1}$.
Note that $S_{n, s-1}$ is $M_{k+1}$-free because the matching number is $s-1\leq k$. 
Since $S_{n, s-1}(s\leq k+1)$ and $\Gamma'$ is a proper subgraph of  $S_{n, k}$, we have $\Gamma=S_{n, k}$; otherwise it contradicts the fact that $\mathcal{P}(G)$ is maximum.
By Claim~\ref{3.2}, we have $G\cong S_{n,k}$, which contradicts the assumption that $G\notin \mathcal{G}^3_{n,k}$. 
 This proves the claim.
\qee

\begin{claim}\label{Claim:H-clique-3}
$H$ is a clique.
\end{claim}
\pf
We shall show that $H$ is a clique. Suppose $x,y \in V(H)$ are not adjacent in $H$. Then $x$ and $y$ are not adjacent in $\Gamma$. Since $G$ is $M_{k+1}$-free, $\Gamma$ is $M_{k+1}$-free. Since $\mathcal{P}(G)$ is maximum and $\mathcal{P}(\Gamma+xy)>\mathcal{P}(G)$, we have $\mu(\Gamma+xy)\geq k+1$. Note that $d_H(x)+d_H(y)\geq 2k+2$. By Lemma~\ref{bc}, $\mu(\Gamma)\geq k+1$. As $\mu(G)\geq \mu(\Gamma)\geq k+1$, a contradiction which proves Claim \ref{Claim:H-clique-3}.
\qee

\begin{claim}\label{cl3-3}
$H$ is a clique with the maximum size in $\Gamma$.
\end{claim}
\pf
Suppose that there exists the other clique, say $H'$ in $\Gamma$ such that $|H'|>|H|$. Then for any vertex $v\in V(H')$, $d_{H'}(v)\geq |H'|-1\geq |H|\geq k+2$.
As $H'$ is a clique in $\Gamma$, any vertex in $H'$ cannot be deleted in $H(\Gamma;k)$, and hence $H'\subseteq H$,
contradicting the fact that $|H'|>|H|$ . This proves the claim.
\qee

Let $r=|V(H)|$.

\begin{claim}\label{claim-inequality-3}
 Then $k+2\leq r\leq 2k$, and so $1\leq 2k+1-r\leq k-1.$
\end{claim}
\pf
As $H=H(\Gamma;k)$ is a clique, $r \geq k+2$. 
Suppose that $r=2k+1$, then there is an $M_k$ in $H$. As $n\geq 2k+2$ and $\Gamma$ is connected, there is an $M_{k+1}$ in $\Gamma$, and so is $G$, a contradiction. If $r\geq 2k+2$, there is an $M_{k+1}$ in $H$, so is $G$. It is a contradiction. Now, we have $k+2\leq r\leq 2k$. 
This proves Claim \ref{claim-inequality-3}.
\qee

\begin{claim}\label{claim-convex-3}
Let $H'=H(\Gamma;2k+1-r)$. Then $H \neq H'$.
\end{claim}
\pf
Suppose $H=H'$. Notice that each vertex from $V(G)\setminus V(H')$ has degree at most $2k+1-r$ in $\Gamma$. So the size of the second largest maximal clique is at most $2k+2-r$. By Lemma~\ref{Ai-Ning} and Claim~\ref{3.2}, $\Gamma$ is a proper subgraph of $K_{2k-r+1}\vee ((n-r)K_1\cup K_{2r-2k-1})$, then
$$
\mathcal{P}(\Gamma)< \mathcal{P}(K_{2k-r+1}\vee ((n-r)K_1\cup K_{2r-2k-1})),
$$ 
a contradiction. This proves the claim.
\qee

Finally, we claim that
$G$ contains a $M_{k+1}$.
Note that $V(H) \subset V(H')$ and $H \neq H'$. We select $x \in V(H) $ and $y \in V(H') \setminus V(H)$ such that $x$ and $y$ are nonadjacent. Note that Claim 1 is also true if we replace $G$ with $\Gamma$. So by Claim 1, there is a $M_{k+1}$ in $\Gamma+xy$. Observe that $d_\Gamma(x)+d_\Gamma(y)\geq r-1+2k-r+2=2k+1$, thus there is an $M_{k+1}$ in $\Gamma$ by Lemma~\ref{bc}, and so is $G$, a contradiction. This proves the theorem.
\end{proof}
\section{Concluding remarks}
\begin{enumerate}
    \item[1.]
    As remarked in \cite{K77}, the classical Theorem \ref{K77} can imply Theorem \ref{K77-2}. One may wonder whether we can deduce Theorem \ref{Thm:Main2} from Theorem \ref{Thm:Main1} or not. Indeed, we can consider the general problem: Let $G_1,G_2$ be two graphs and let $\mathcal{C}(G_1)$, $\mathcal{C}(G)$ be feasible parameters. Suppose $\mathcal{C}(G_1)>\mathcal{C}(G_2)$. Is it always true that $\mathcal{C}(G_1\vee K_1)>\mathcal{C}(G_2\vee K_1)$? The answer is negative. In fact, it is false when we consider just a problem under the spectral radius condition. Consider the following example: let $G_1=K_3\vee 5K_1$ and $G_2=K_1\vee(K_5+2K_1)$.
    Then $\lambda(G_1)=5<\lambda(G_2)=5.0695$,
    but $\lambda(G_1\vee K_1)=\lambda(K_4\vee 5K_1)=6.2170>\lambda(G_2\vee K_1)=\lambda(K_2\vee(K_5+2K_1))=6.1970$. The following problem is still wide open.
\begin{prob}
Let $G_1,G_2$ be two graphs with $\lambda(G_1)\geq \lambda(G_2)$. Determine which graphs $G_1,G_2$ satisfying the property that $\lambda(G_1\vee K_1)\geq \lambda(G_2\vee K_1)$.
\end{prob}
\item[2.] Theorem~\ref{Thm:Main1} tells us how a feasible graph parameter behaves under the constraints of the length of circumference for $2$-connected graphs. It is natural to ask what if for a connected graph or a general graph.
\item[3.] Compared with the results for a feasible graph parameter, when $\mathcal{P}$ is weakly feasible, we can just determine the extremal values but the extremal graphs. Can the corresponding extremal graphs also be determined by a more accurate discussion?
\end{enumerate}


\end{document}